\documentclass[12pt]{amsart}
\usepackage{amsmath}
\usepackage{amsthm}
\usepackage{amsfonts}
\usepackage{amssymb}
\newtheorem{Theorem}{Theorem}[section]
\newtheorem{Proposition}[Theorem]{Proposition}
\newtheorem{Lemma}[Theorem]{Lemma}
\newtheorem{Corollary}[Theorem]{Corollary}
\theoremstyle{definition}

\theoremstyle{remark}

\numberwithin{equation}{section}

\newcommand{\Z}{{\mathbb Z}}
\newcommand{\R}{{\mathbb R}}
\newcommand{\C}{{\mathbb C}}
\newcommand{\N}{{\mathbb N}}

\newcommand{\J}{{\mathcal J}}
\newcommand{\RR}{{\mathcal R}}
\renewcommand{\H}{{\mathcal H}}
\begin{document}

\title[Reflectionless Jacobi matrices]{Approximation results
for reflectionless Jacobi matrices}

\author{Alexei Poltoratski}

\address{Mathematics Department\\
Texas A\&M University\\
College Station, TX 77843}

\email{alexei@math.tamu.edu}

\urladdr{www.math.tamu.edu/$\sim$alexei.poltoratski/}

\author{Christian Remling}

\address{Mathematics Department\\
University of Oklahoma\\
Norman, OK 73019}

\email{cremling@math.ou.edu}

\urladdr{www.math.ou.edu/$\sim$cremling}

\date{May 10, 2010}



\thanks{AP's work supported by NSF grant 0800300; CR's work supported
by NSF grant DMS 0758594}
\begin{abstract}
We study spaces of reflectionless Jacobi matrices. The main theme is
the following type of question: Given a reflectionless Jacobi matrix,
is it possible to approximate it by other reflectionless and, typically, simpler
Jacobi matrices of a special type? For example, can we approximate by periodic
operators?
\end{abstract}
\maketitle
\section{Introduction}
We continue our study of reflectionless Jacobi matrices, which was begun in
\cite{PolRem,Remac}.
In this paper, by a \textit{Jacobi matrix }we mean a
bounded, self-adjoint operator $J$ on $\ell^2(\Z)$ (always whole line!) which
acts as follows:
\[
(Ju)(n) = a(n)u(n+1)+a(n-1)u(n-1)+b(n)u(n)
\]
Alternatively, one can represent $J$ by the following tridiagonal matrix with respect to
the standard basis of $\ell^2(\Z)$:
\[
J = \begin{pmatrix} \ddots & \ddots & \ddots &&&& \\ & a(-2) & b(-1) & a(-1) &&&\\
&& a(-1) & b(0) & a(0) && \\ &&& a(0) & b(1) & a(1) & \\ &&&& \ddots & \ddots & \ddots
\end{pmatrix}
\]
Here, $a(n)\ge 0$ and $b(n)\in\R$, and we also assume that $a,b\in\ell^{\infty}(\Z)$.
(Usually, one insists that $a(n)>0$, but for reasons of formal elegance, our convention
seems preferable here.)
The set of all such Jacobi matrices will be denoted by $\mathcal J$.

A Jacobi matrix is called \textit{reflectionless }on a bounded Borel
set $B\subset\R$ if for all $n\in\Z$,
\[
\textrm{\rm Re }g_n(t) = 0 \quad \textrm{\rm for almost every }t\in B ,
\]
where $g_n(z)=\langle \delta_n , (J-z)^{-1} \delta_n \rangle$, and, as usual,
$\delta_n(j)=1$ if $j=n$ and $=0$ if $j\not= n$. It will be
convenient to denote the set of reflectionless (on $B$) Jacobi matrices by
$\mathcal R (B)$.

Reflectionless Jacobi matrices are of special interest because they provide
the basic building blocks for arbitrary Jacobi matrices with non-empty absolutely
continuous spectrum. See \cite{Remac}.

We will usually work with topological spaces of reflectionless Jacobi matrices
rather than single operators. These spaces will always be contained in
\[
\J_R = \{ J\in\J: \|J\|\le R \}
\]
for some $R>0$.
The topology we are interested in can be described
as product topology on the coefficients; it is induced by the metric
\[
d(J,J') = \sum_{n=-\infty}^{\infty} 2^{-|n|} \left( |a(n)-a'(n)| + |b(n)-b'(n)| \right) .
\]
This is also the topology that is induced on $\J_R$ by the weak or strong operator topology.
This topology is by far the most useful one for the questions we are interested
in here for many reasons, not the least of which are
its smooth interaction with other natural topologies and the fact that it makes
$\J_R$ a compact space.
These two themes will play a prominent role throughout this paper.

A Jacobi matrix is called \textit{periodic }if its coefficients have this property,
that is, $a(n+p)=a(n)$, $b(n+p)=b(n)$ for all $n\in\Z$ and some period $p\in\N$.
It is well known that periodic Jacobi matrices are reflectionless on their spectrum, which
is a \textit{finite gap set }(a union of finitely many compact intervals).
Our first main result says that, conversely, any $J\in \mathcal R (B)$ can be approximated
by periodic operators with spectrum almost equal to $B$. More precisely:
\begin{Theorem}
\label{T1.1}
$J\in \mathcal R (B)$ if and only if there are $R>0$ and
periodic Jacobi operators $J_n\in\J_R$ with
spectra $\sigma(J_n)=P_n$, such that $d(J_n,J)\to 0$ and $|B\Delta P_n|\to 0$.
\end{Theorem}
Here, $B\Delta P$ denotes the symmetric difference of $B$ and $P$ (the set of all points
that are in one set, but not in the other), and $|\cdot|$ refers to Lebesgue measure.

The main interest lies in the ``only if'' part; the converse
statement is an immediate
consequence of the fact that $\mathcal R(B)\cap\J_R$ is a compact set
with respect to $d$; see Proposition \ref{P3.1}(a) below for more details.
We have stated this converse for completeness.

Our original motivation for Theorem \ref{T1.1} came from Corollary \ref{C1.1} below.
Let us first formulate an abstract version of this statement.
We define the action of the (left) shift $S$ on $J\in\mathcal J$ in the obvious way:
$J'=SJ$ has coefficients $a'(n)=a(n+1)$, $b'(n)=b(n+1)$.
\begin{Theorem}
\label{T1.2}
Let $\varphi: \mathcal J\to \mathcal J$ be a continuous (with
respect to $d$, as always) map that preserves spectra and commutes with the shift:
$\sigma(\varphi(J))=\sigma(J)$, $\varphi S=S\varphi$. Then $\varphi(J)\in\mathcal R (B)$
whenever $J\in \mathcal R (B)$.
\end{Theorem}
The application we have in mind here is to Toda flows, as spelled out in Corollary \ref{C1.1} below.
Let us quickly recall the basic setup here;
for more background information on Toda flows, see, for example, \cite{Teschl}.
For any (real valued) polynomial $p$, we can define an associated Toda flow, as follows.
Let $p_a(J)$ be the anti-symmetric part of $p(J)$: write $p(J)$ as a matrix with respect
to the standard basis of $\ell^2(\Z)$, change the signs in the lower triangular part and
delete the diagonal to obtain $p_a(J)$. Then the differential equation
\[
\dot{J} = [p_a(J),J]
\]
defines a global flow on $\J$, which we call the \textit{Toda flow }associated with $p$.
\begin{Corollary}
\label{C1.1}
If $J(0)\in\mathcal R (B)$, then
$J(t) \in \mathcal R (B)$ for any Toda flow and all times $t\in\R$.
\end{Corollary}
This follows because the \textit{Toda maps }$J=J(0)\mapsto J(t)$ (for fixed $t\in\R$ and
polynomial $p$) are known to have all the properties required in Theorem \ref{T1.2}.
(The continuity with respect to $d$ is perhaps not addressed explicitly in the existing
literature, but this is easy to establish. It will also be discussed in a forthcoming publication
\cite{RemToda}, and, in any event, this discussion would take us too far afield here.)

In the second part of this paper, we will try to analyze the spaces
\[
\RR_0(K) = \{ J\in\RR (K): \sigma(J)\subset K \} ,
\]
for compact subsets $K\subset\R$. If $K$ is essentially closed, then
all $J\in\RR_0(K)$ will in fact satisfy $\sigma(J)=K$.

For any compact $K\subset\R$, we have:
\begin{Proposition}
\label{P1.1}
$\RR_0(K)$ is compact.
\end{Proposition}
As usual, this statement refers to the topology that is induced by the metric $d$.
Proposition \ref{P1.1} is not a new result; compare, for example, \cite{Kot, Remac}.
However, the Proposition will also be a consequence of our
Proposition \ref{P3.1} below, so we will briefly discuss its proof here.

It is also clear that $\RR_0(K)$ is shift invariant, and, ideally, one would like to
understand the dynamical system $(\RR_0(K),S)$. This task has been accomplished for
so-called homogeneous sets $K$ by Sodin-Yuditskii \cite{SodYud}; the special case of
a finite gap set $K$ is classical and has been studied in very great detail.
Beyond the Sodin-Yuditskii result, very little is known at present.

In this paper, we set ourselves the more modest task of analyzing the $\RR_0(K)$
for compact $K\subset\R$ as a collection of \textit{topological spaces }rather than
\textit{dynamical systems, }but we would like to
do this for rather general sets $K$. In this endeavor, the main difficulty comes from a
possible singular part of spectral measures and related measures on $K$.
If there can't be any such singular part, the analysis becomes much easier.
For example, the following statement holds.
\begin{Theorem}
\label{T1.3}
Suppose that $\rho_s(K)=0$ for all $\rho\in\H(K)$. Then $\RR_0(K)$ is homeomorphic
to the $N$-dimensional torus $\mathbb T^N$, where $N\in\N_0\cup\{\infty\}$
is the number of bounded components (``gaps'') of $K^c$.
\end{Theorem}
Here, $\H(K)$ denotes the collection of measures
that come from the so-called $H$ functions of the $J\in\RR_0(K)$. This $H$ function is defined
as the negative reciprocal of the Green function at $n=0$: $H(z)=-1/g_0(z)$.
$H$ is a Herglotz function and thus there is a unique associated measure.
Please see Section 2 for a more detailed discussion of these definitions.

The hypothesis of Theorem \ref{T1.3} will be satisfied
by many sets $K$. For example, Theorem \ref{T1.3} applies to all
weakly homogeneous sets $K$ \cite[Corollary 2.3]{PolRem}; see also \cite{SodYud}.

If $N=\infty$ in Theorem \ref{T1.3}, then $\mathbb T^{\infty}=S^{\N}$ is defined as the infinite
Cartesian product of countably many copies of the circle $S$,
and we use product topology on this space. If $N=0$ (so $K$ is an interval), then we just define
$\mathbb T^0$ to be a single point. This of course is a well known special case: if $K$
is a compact interval, then the only Jacobi matrix that is reflectionless on $K$ and has
$K$ as its spectrum is the one with the appropriate constant coefficients.

While we haven't seen it in print in this form, we don't want to claim much credit for
Theorem \ref{T1.3}. There is a natural way of setting up a bijection between $\RR_0(K)$ and
$\mathbb T^N$, which has been used
by many authors \cite{Craig,Kot,SodYud,Teschl}, and what we add here
is the observation that this map is continuous with respect to the chosen topologies.

Moving on to the more original (we hope) parts of our discussion of the spaces $\RR_0(K)$,
we would now like to
understand how $\RR_0(K)$ changes with $K$. The spaces $\RR_0(K)$ are compact subsets of
$\J$, so it seems natural to use Hausdorff distance to compare their coarse structure,
what they look like if viewed from a distance. Recall that the Hausdorff distance between two
compact, non-empty subsets $A,B$ of a metric space is defined as
\[
h(A,B) = \max \{ \sup_{x\in A} d(x,B), \sup_{y\in B} d(y,A) \} .
\]
We would now like to know what conditions need to be imposed on $K,K'$ if we want
$\RR_0(K)$, $\RR_0(K')$ to be close in Hausdorff distance. For instance, is it always
possible to approximate a potentially complicated space $\RR_0(K)$ by a simpler space,
say $\RR_0(K')$ for a finite gap set $K'$?

More formally, we can observe that $K\mapsto \RR_0(K)$ is an injective map that is defined
on non-empty, compact sets $K\subset\R$, so we can pull back the Hausdorff metric and
ask for a description of the metric on the space
$\{ K\subset\R: K\textrm{ compact, }K\not=\emptyset\}$ that is obtained in this way.
More precisely, we would like to be able to write down an equivalent metric (one that
generates the same topology).

We will not be able to completely answer this question in this paper, but we can
report on some progress. Based on what we do below, it in fact seems reasonable to
conjecture that a possible choice for the sought metric is
\begin{equation}
\label{defdelta}
\delta(K,K') = h(K,K') + |K\Delta K'| .
\end{equation}
It is easy to see that this indeed defines a metric $\delta$ on
the non-empty compact subsets of $\R$; we will
discuss this fact in Proposition \ref{P5.1} below.

Notice that $\delta$ generates a stronger topology on $\{ K \}$
than its first summand Hausdorff distance $h$.
It is actually clear that $\{ K\subset [-R,R] \}$ with the correct metric,
whatever it may be, is not a compact space, and this immediately rules out Hausdorff distance
as the answer to our question.
For example, one can show that
if $K_n =\bigcup_{j=1}^n [j/n,j/n+1/n^2]$, say, then with respect to Hausdorff distance,
\[
\lim_{n\to\infty} \RR_0(K_n) = \{ J: \sigma(J)\subset [0,1] \} ,
\]
which is not equal to $\RR_0(K)$ for any $K$.

Here's one possible formulation of what we will actually prove.
\begin{Theorem}
\label{T1.4}
Let $K, K_n\subset[-R,R]$ be non-empty, compact sets,
and abbreviate $\Omega=\RR_0(K), \Omega_n=\RR_0(K_n)$.\\
(a) If $h(\Omega_n,\Omega)\to 0$, then $\delta(K_n,K)\to 0$.\\
(b) Suppose that $\rho_{sc}(K)=0$ for all $\rho\in\H(K)$. Then
$\delta(K_n,K)\to 0$ implies that $h(\Omega_n,\Omega)\to 0$.
\end{Theorem}
This is an incomplete answer to the question we posed because of the additional assumption
in part (b) (which is similar to, but considerably weaker than the assumption we made in
Theorem \ref{T1.3} above).
We can be somewhat more specific here: Notice, first of all, that for two compact
sets to be close in Hausdorff distance, we have to be able to come close to each point from one
set by an element of the other set, and vice versa. However, as we'll explain below, at the
beginning of Section 5, the only real issue here is the question of whether we will be able
to approximate any $J\in\Omega$ by $J_n\in\Omega_n$. Indeed, in Theorem \ref{T1.5} below,
the additional assumption is required only in part (b), when we try to approximate
$J\in\Omega$.
\begin{Theorem}
\label{T1.5}
Let $K_n,K\subset[-R,R]$ be non-empty compact sets, and
suppose that $\delta(K_n,K)\to 0$. Then:\\
(a)
\[
\lim_{n\to\infty}\sup_{J\in\Omega_n} d(J,\Omega) = 0
\]
(b) If $J\in\Omega$ and the associated measure satisfies $\rho_{sc}(K)=0$, then
\[
\lim_{n\to\infty} d(J,\Omega_n) = 0 .
\]
\end{Theorem}
As we'll discuss in Section 5, Theorem \ref{T1.4}(b)
will be a quick consequence of Theorem \ref{T1.5}. The crucial ingredient
to the proof of Theorem \ref{T1.5}, in turn, will be the following approximation result,
which could be of some independent interest.
\begin{Theorem}
\label{T4.1}
Let $J\in\RR_0(K)$ and suppose that the associated $\rho\in\mathcal H(K)$
satisfies $\rho_{sc}(K)=0$. Then there are
$J_n\in\RR_0(K)$ with $\rho^{(n)}_s(K)=0$, so that $d(J_n,J)\to 0$.
\end{Theorem}

We organize this paper as follows. Section 2 discusses basic material from
inverse spectral theory, but in a version that is tailor made for the investigation
of specifically reflectionless operators. We also prove Theorem \ref{T1.3} in this section.
Theorems \ref{T1.1} and \ref{T1.2} are proved in the Section 3.
In Section 4, we discuss Theorem \ref{T1.4}(a). Section 5 has the proofs
of Theorems \ref{T1.4}(b), \ref{T1.5}, and \ref{T4.1}.

In the main body of this paper, that is, in this introduction and Sections 2--5,
we assume that the sets $B$ and $K$ are of positive Lebesgue measure. This is really
the relevant case here since we want to understand aspects of the absolutely continuous spectrum.
However, from a formal point of view, our results remain correct if $|B|=0$ or $|K|=0$.
Note that if $|B|=0$, then $\RR(B)=\J$, and $\RR_0(K)=\{ J: \sigma(J)\subset K\}$ if $|K|=0$.
The arguments needed in these cases are easier than, but also different from those for
the positive measure case. Therefore, we very briefly discuss them separately, in the final section.
In Sections 1--5, we always assume that $|B|>0$, $|K|>0$.

In fact, we can strengthen Theorems \ref{T1.4}, \ref{T1.5}, and \ref{T4.1} if $|K|=0$,
and we do obtain complete answers in this case. For example, if $|K|=0$, then
$h(\Omega_n,\Omega)\to 0$ if and only $\delta (K_n,K)\to 0$, which, in this case,
happens if and only if $h(K_n,K)\to 0$.
\section{Spectral data for reflectionless Jacobi matrices}
In this section, we review and develop further basic tools that will allow
us to conveniently describe reflectionless Jacobi matrices in terms of carefully
chosen spectral data. This material will be fundamental for everything we do in
this paper. See \cite{Craig,Kot,Remac,Teschl} for earlier applications of this
basic method, and also \cite{GesSimxi} for related uses of the $\xi$ function.
Our presentation here follows \cite[Sections 5, 6]{Remac} very closely,
with some additional material added. We only sketch most of the proofs here and refer
the reader to this reference for full details.

This discussion will come in three parts. In the first part, we introduce convenient (for our
purposes) spectral data for arbitrary bounded Jacobi matrices. We then modify these data
to obtain a related param\-e\-tri\-zation of $\RR(B)$, and, in the final part, we introduce still
another variant of this, which will be particularly useful when we discuss the spaces
$\RR_0(K)$.

Let $J\in\mathcal J$, and consider the Herglotz function
\[
H(z) = -\frac{1}{g(z)} ,
\]
where $g(z)=\langle \delta_0, (J-z)^{-1} \delta_0 \rangle$. Let $\xi$ be the Krein function
of $H$, that is,
\[
\xi(t) = \frac{1}{\pi}\lim_{y\to 0+} \textrm{Im }\ln H(t+iy) .
\]
Since $H(z)$ is never zero on the upper half plane $\C^+$, we can take a holomorphic
logarithm $\ln H(z)$. In fact, $H(z)$ is a Herglotz function, so we can demand that
the imaginary part of this logarithm lies in $(0,\pi)$. The limit defining $\xi$ exists
for (Lebesgue) almost every $t\in\R$; we view $\xi$ as an element of $L^{\infty}(\R)$,
and in fact we have that $0\le\xi(t)\le 1$.
The Krein function is an exceedingly useful tool here, mainly because of the following fact:
\begin{Proposition}
\label{P2.1}
If $J\in\RR(B)$, then $\xi =1/2$ almost everywhere on $B$.
\end{Proposition}
\begin{proof}
This is obvious because if $J$ is reflectionless on $B$, then, by definition,
$g$ and thus also $H$ are purely imaginary on this set.
\end{proof}

We can recover $H$ from its Krein function $\xi$; in fact, we have the explicit formula
\begin{equation}
\label{2.5}
H(z) = (z+R)\exp\left( \int_{-R}^R \frac{\xi(t)\, dt}{t-z} \right) .
\end{equation}
Here, $R$ is chosen so large that $\|J\|\le R$; we have also made use of the asymptotic formula
$H(z)=z+O(1)$ as $|z|\to \infty$ to determine an otherwise unknown constant.

However, we cannot, in general, recover the Jacobi matrix $J$ from $H$ or, equivalently,
$\xi$. We must introduce additional spectral data. We proceed as follows.
Write down the Herglotz representation for $H$:
\begin{equation}
\label{2.6}
H(z) = z+A + \int_{(-R,R)} \frac{d\rho(t)}{t-z}
\end{equation}
Here, $\rho$ is a finite and (obviously) compactly supported Borel measure.
The constant $A$ can be identified in
terms of the previously used data as
\[
A = R - \int_{-R}^R \xi(t)\, dt .
\]
Furthermore, we can use the half line spectral measures $\rho_{\pm}$
to decompose $\rho$ as follows:
\[
\rho = a(0)^2 \rho_+ + a(-1)^2\rho_-
\]
Please see \cite[Chapter 2]{Teschl} for this information and precise definitions
of $\rho_{\pm}$. (Warning: These measures are called $\widetilde{\rho}_{\pm}$ in
\cite{Teschl}.)

It will be convenient to introduce $\nu_+=a(0)^2\rho_+$,
$\nu_-=a(-1)^2\rho_-$ and rewrite this as
\begin{equation}
\label{2.12}
\rho = \nu_+ + \nu_- .
\end{equation}
We will also refer to these measures $\nu_{\pm}$ as half line spectral measures. Note that
$\nu_+$ or $\nu_-$ or both of these can be the zero measure.

Every pair of positive finite measures $(\rho_+,\rho_-)$ with $\rho_{\pm}(\R)=1$
is admissible here as a pair of half line spectral measures; we thus obtain the
following (almost) parametrization of (bounded) Jacobi matrices:
With each $J\in\J$, associate its $\xi$ function
and also its positive half line spectral measure $\nu_+$. As just discussed,
these data have the property that if $\rho$ is defined via \eqref{2.5},
\eqref{2.6}, and then $\nu_-$ via \eqref{2.12},
then $\nu_-$ is a positive (finite) measure (possibly the zero measure).

Conversely, suppose that such a pair $(\xi,\nu_+)$ is given.
In other words, if again $\rho$ is the measure associated with
$\xi$ and \eqref{2.12} is used to define $\nu_-$, then $\nu_-$ is a positive measure.
Then there exists a Jacobi matrix $J$ that has $(\xi,\nu_+)$ as its spectral data;
moreover, $J$ will be unique if both $\nu_+$ and $\nu_-$ have infinite supports
(in particular, neither is the zero measure).
This follows from the usual inverse spectral theory for half line problems:
If $\nu_{\pm}(\R)>0$, then we can of course recover
$\rho_{\pm}$ (and $a(0)$, $a(-1)$) from $\nu_{\pm}$, because these measures
are probability measures. Then the measures $\rho_{\pm}$ determine all the
remaining coefficients except for $b(0)$; here, we use the assumption that
$\rho_{\pm}$ are not finitely supported.
Moreover, $b(0)=-A$, with $A$ from \eqref{2.6}, which is determined by $\xi$.
See \cite[Chapter 2]{Teschl} and \cite[Section 5]{Remac} for more background information
on this procedure.

If $\nu_+$ or $\nu_-$ is finitely supported, then $a(n)=0$ for some $n$ and only the
coefficients up to this index are determined by $(\xi,\nu_+)$.
For instance, to give an extreme example,
if $\xi=\chi_{(-R,b)}$, then $H(z)=z-b$ and $\rho=\nu_+=\nu_-=0$, and we can only conclude
that $a(-1)=a(0)=0$, $b(0)=b$.

We are interested in approximation properties in this paper, so it is essential to
use parametrizations with good continuity properties. Indeed, we have:
\begin{Proposition}
\label{P2.5}
Fix $R>0$.
The map that sends $J\in\J_R$
to $(\xi,\nu_+)$ is a continuous map between compact metric spaces.
Here, we use the weak-$*$ topology for both $\xi\, dt$ and $d\nu_+$ on the image.
\end{Proposition}
More precisely, we make the image a (compact)
metric space in the following way. Fix a metric $D$ that induces
the weak-$*$ topology on the Borel measures $\mu$ on $[-R,R]$ with $\|\mu\|\le C$. Then let
\begin{equation}
\label{2.24}
d((\xi,\nu_+),(\xi',\nu'_+))=D(\xi\, dt, \xi'\, dt)+D(\nu_+, \nu'_+) .
\end{equation}
\begin{proof}[Sketch of proof]
The spectral measure $\nu_+$ depends continuously on $J$
with respect to the chosen topologies; this is a well known basic fact.
See, for example, \cite[Lemma 3.2]{Remac}. Moreover,
if $J_n\to J$, then $H_n(z)\to H(z)$ locally uniformly on $z\in\C^+$. But then
$\ln H_n(z)$ also converges, to $\ln H(z)$, and these new Herglotz functions have associated
measures $\xi_n\, dt$ and $\xi\, dt$, respectively, so
we now obtain the asserted convergence $\xi_n\, dt\to\xi\, dt$ in weak-$*$ sense.
Compare \cite[Theorem 2.1]{Remac} for this last step.
\end{proof}

If $J\in\RR(B)$ for some positive measure Borel set $B\subset\R$, then, as we saw above,
$\xi=1/2$ on $B$, and, moreover, $\nu_+$ also has to satisfy a related condition. We have that
\begin{equation}
\label{2.14}
d\nu_+(t) = f(t)\, d\rho(t) ,
\end{equation}
where $f$ is a Borel function that satisfies $0\le f\le 1$. Of course, so far this is just
a rephrasing of what \eqref{2.12} says about $\nu_+$. However, $J\in\RR(B)$ also implies
that $f(t)=1/2$ for Lebesgue almost every
$t\in B$. See \cite[Corollary 5.3]{Remac}. These data $(\xi, f)$ provide a complete
parametrization of $\RR(B)$. Here, $f$ is thought
of as an element of $L^1(\R, d\rho)$, where $\rho$ is the measure associated with $\xi$.

So, if $(\xi,f)$ as above is given (that is, $\xi=f=1/2$ almost everywhere on $B$ and
$0\le\xi, f\le 1$), then there is a unique $J\in\RR(B)$ whose spectral data are $(\xi,f)$.
This time, there are no strings attached: We know that $\chi_B\, dt$ is absolutely continuous
with respect to $d\rho$, and, since $f=1/2$ Lebesgue almost everywhere on $B$, also with respect
to $d\nu_{\pm}$, so these measures have infinite supports.
See again \cite[Corollary 5.3]{Remac} for the details.
The correspondence $J\leftrightarrow (\xi, f)$ is continuous in both directions:
\begin{Proposition}
\label{P2.2}
Fix $R>0$ and a Borel set $B\subset (-R,R)$ with $|B|>0$.
The map that sends $J\in\RR(B)\cap\J_R$ to $(\xi,f)$ is a homeomorphism onto its image.
On this image, we use the weak-$*$ topology for the measures $d\nu_+=f\, d\rho$; both
spaces are compact metric spaces.
\end{Proposition}
We could, of course, use the metric from \eqref{2.24} again (with $d\nu_+=f\, d\rho$), but it is
also possible, as indicated, to just use
\[
d((\xi,f),(\xi',f'))=D(f\, d\rho, f'\,d\rho')
\]
instead. This works because a
reflectionless Jacobi matrix $J\in\RR(B)$ is already determined by its half line restriction
and thus also by $d\nu_+=f\, d\rho$.
Compare, for example, \cite[Proposition 4.1]{Remac}.
So it indeed suffices to work with the half line spectral measures $\nu_+$ when defining $d$:
if $d((\xi,f),(\xi',f'))=0$, then $\nu_+=\nu'_+$, which, as just explained, implies that
$J=J'$; in particular $\xi=\xi'$ and thus $(\xi,f)=(\xi',f')$, as required.
\begin{proof}
We already know that this map is injective, and it is obviously continuous, being a restriction
of the map from Proposition \ref{P2.5}. Moreover, $\RR(B)\cap\J_R$ is compact
(see again \cite[Proposition 4.1]{Remac}), and a continuous bijection between compact metric spaces
automatically has a continuous inverse.
\end{proof}

We now specialize further and seek a parametrization of
$\RR_0(K)$, for a compact set $K\subset\R$ with $|K|>0$.
We will use the symbol $\mathcal H(K)$ to denote
the collection of all measures $\rho$ from \eqref{2.6} that correspond to some $J\in\RR_0(K)$.
This is the set that was referred to in the formulation of Theorems \ref{T1.3} and \ref{T1.4}(b).
Similarly, we let $X(K)$ be the collection of the $\xi$ functions of the $J\in\RR_0(K)$.

By the spectral theorem, $g$ has a representation of the type $g(z)=\int_K \frac{d\mu(t)}{t-z}$, with
a probability measure $\mu$. This implies that on each bounded component $(a,b)\subset K^c$,
$\xi$ is a step function that jumps from $0$ to $1$:
\begin{equation}
\label{2.4}
\xi(t) = \chi_{(\mu,b)}(t) \quad (a<t<b)
\end{equation}
for some $\mu\in [a,b]$. These intervals $(a,b)$ will also be referred to as \textit{gaps }(of $K$).
Notice that the parameters $(\mu_j)$ (one for each gap) determine $\xi$. Indeed, we always
have that $\xi=1/2$ on $K$ and $\xi=1$ to the left of $K$ and $\xi=0$ to the right of $K$,
so it suffices to specify $\xi$ on each gap to have a complete definition of a $\xi\in X(K)$.

The correspondence
between $\mu$ and $\xi$ that is obtained in this way
is again a homeomorphism with respect to the natural topologies:
\begin{Proposition}
\label{P2.3}
Fix a compact, non-empty set $K\subset\R$, with gaps $(a_j,b_j)$.
The map $X(K)\to \prod [a_j,b_j]$ that sends $\xi$ to $\mu$ is a homeomorphism
if we use the product topology on the second space.
\end{Proposition}
On $X(K)$, we of course use the weak-$*$ topology for $\xi\, dt$, as before. Both spaces
are in fact compact metric spaces, a possible choice for the metric on the second space is
\[
d(\mu,\mu') = \sum_{j=1}^{\infty} |\mu_j-\mu'_j| .
\]
\begin{proof}[Sketch of proof]
It is clear from our discussion above that the map $\xi\mapsto\mu$
is bijective, and continuity (in both directions) is easy to confirm.
\end{proof}

Every $\rho\in\mathcal H(K)$ is of the following form:
\begin{equation}
\label{2.1}
d\rho = \chi_K F\, dt + \chi_K\, d\rho_s + \sum w_j\delta_{\mu_j}
\end{equation}
Here, the sum is taken over all gaps for which $a_j<\mu_j<b_j$. We can be sure
that there will indeed be a point mass at all these $\mu_j$ (in other words, $w_j>0$)
because $\xi$ jumps from $0$ to $1$ at these points, so we can refer
to a basic criterion for the existence of point masses. See \cite[pg.\ 201]{MP} or
\cite[Lemma 2.4]{PolRem}. Similarly, $\rho$ can not have an additional singular part
off $K$; in fact, $H$ has a holomorphic continuation through every interval $I\subset K^c$
with $\mu_j\notin I$.
We would like to again remind the reader that $F>0$ almost everywhere on $K$, because
$\pi F=\textrm{Im }(-1/g)=|1/g|$, and this can not be zero on a positive measure set.

The function $f$ from above, which determines $\nu_+$ through \eqref{2.14},
now has to satisfy $f=1/2$ Lebesgue almost everywhere on $K$
and $f(\mu_j)=0$ or $1$ for all $j$ that
contribute to the sum in \eqref{2.1}. The first property was noted above for general $J\in\RR(K)$,
and the second property follows quickly from the additional requirement that $\sigma(J)\subset K$;
see \cite[Section 6]{Remac} for a more detailed discussion.

In other words,
\begin{equation}
\label{2.2}
d\nu_+ = \frac{1}{2}\chi_K F\, dt + \chi_K g\, d\rho_s + \sum\sigma_j w_j\delta_{\mu_j} ,
\end{equation}
where $g$ is a Borel function with $0\le g\le 1$ and $\sigma_j\in \{ 0, 1\}$.
We will use $(\mu,\sigma,g)$, with $\mu=(\mu_j)_j$,
$\sigma=(\sigma_j)_j$, and $g$ from \eqref{2.2},
as the spectral data for $J\in\RR_0(K)$. As above, we view $g$ as an element of
$L^1(K, d\rho_s)$, that is, we identify $g$'s that agree almost everywhere on $K$ with
respect to $\rho_s$. Frequently, no $\rho\in\mathcal H(K)$ can have a singular part on $K$, and then
we can discard $g$ altogether, and we parametrize $\RR_0(K)$ by just $(\mu,\sigma)$.

It is again true that,
conversely, any set of parameters $(\mu,\sigma,g)$ of this type
will correspond to a unique $J\in\RR_0(K)$. Let us describe the
corresponding procedure one more time: Given $(\mu,\sigma,g)$,
we first of all define $\xi$ by \eqref{2.4} and then, as usual, construct
$\rho$ from \eqref{2.5}, \eqref{2.6}. Then \eqref{2.2} gives us
$\nu_+$, and then $\nu_-$ is obtained from \eqref{2.12}. From
the half line spectral measures, we can recover a unique $J$ if $\nu_{\pm}$ have
infinite supports, and this is automatic here.
This Jacobi matrix $J$ will lie in $\RR_0(K)$.

This parametrization of $\RR_0(K)$ again has the
desired continuity properties:
\begin{Proposition}
\label{P2.4}
Fix a compact, non-empty set $K\subset\R$ of positive Lebesgue measure.
The map that sends $J\in\RR_0(K)$ to $(\mu,\sigma,g)$ is a homeomorphism from
$\RR_0(K)$ onto its image. Here, we use the metric from Proposition \ref{P2.2}:
\[
d((\mu,\sigma,g),(\mu',\sigma',g'))=D(\nu_+,\nu'_+) ,
\]
where $\nu_+=\nu_+(\mu,\sigma,g)$ is as in \eqref{2.2}, and
the metric $D$ induces the weak-$*$ topology on these measures.
\end{Proposition}
\begin{proof}
As explained above, this map is a bijection onto its image, and it is clearly
continuous because it can be thought of as a restriction of the map from Proposition \ref{P2.2}.
Since $\RR_0(K)$ is compact, the image is compact, too, and continuity of the inverse map
is automatic.
\end{proof}

Note that the individual parts of the decomposition \eqref{2.2} of
$\nu_+$ do not necessarily depend continuously on $J\in\RR_0(K)$.
For example, Theorem \ref{T4.1} shows
that there are situations where $J_n\to J$ and $(\nu_+)_s(K)>0$, but $(\nu_{n,+})_s(K)=0$.

As our first application of the material discussed in this section, let us show that
Theorem \ref{T1.3} is a direct consequence of the parametrization from Proposition \ref{P2.4}.
\begin{proof}[Proof of Theorem \ref{T1.3}]
We now assume
that $\rho_s(K)=0$ for all $\rho\in\mathcal H(K)$, so Proposition \ref{P2.4} in fact provides
a homeomorphism that sends
$J\in\RR_0(K)$ to $(\mu,\sigma)$. We want to
map these data onto $\mathbb T^N$, with $N\in\N\cup\{\infty \}$
being the number of gaps. We will do this componentwise, by mapping
$(\mu_j,\sigma_j)$ for fixed $j$ to a copy of the unit circle $S$.
Before we give the precise definition of this map, let us observe the following: Suppose
that $J_n,J\in\RR_0(K)$, $J_n\to J$. Then $\mu_j^{(n)}\to\mu_j$ as $n\to\infty$ for every
fixed $j$, and if $\mu_j\not= a_j,b_j$, then also $\sigma_j^{(n)}\to\sigma_j$.
Recall in this context that the parameter $\sigma_j$ is not used
if $\mu_j=a_j$ or $\mu_j=b_j$.

To prove this, we can argue as follows: The convergence $\mu_j^{(n)}\to\mu_j$ is an
immediate consequence of Propositions \ref{P2.5}, \ref{P2.3}.
We also know from the proof of Proposition \ref{P2.5} that
$\rho_n\to\rho$.
In our current situation,
with $\rho_s(K)=0$, we can rewrite \eqref{2.2} as
\begin{equation}
\label{2.18}
d\nu_+=\frac{1}{2}\, d\rho + \frac{1}{2} \sum (2\sigma_j-1)w_j\delta_{\mu_j} .
\end{equation}
Now notice that if $\mu_j\not= a_j,b_j$, then we must also have that
\begin{equation}
\label{2.19}
w_j = \lim_{n\to\infty} w_j^{(n)}
\end{equation}
here. Indeed, $w_j=\int f\, d\rho$ and $w_j^{(n)}=\int f\, d\rho_n$
for all large $n$ for a suitable continuous test function $f$ that is supported by the
gap $(a_j,b_j)$ and equal to $1$ in a neighborhood of $\mu_j$, so \eqref{2.19}
follows from the weak-$*$ convergence $\rho_n\to\rho$.
We also have that $\nu_{n,+}\to\nu_+$, so \eqref{2.18}
implies that $\sigma_j^{(n)}\to\sigma_j$, as claimed.

We are now ready to describe the sought map from our parameter space onto
$\mathbb T^N$. More precisely, we will give the inverse map.
Let $(z_j)_j=(e^{i\pi x_j})_j\in\mathbb T^N$. For fixed $j$, let $F_j$ be the following map:
\[
F_j(e^{i\pi x_j}) = \begin{cases} (a_j+ x_j(b_j-a_j) , 1) & 0< x_j < 1 \\
(a_j - x_j (b_j-a_j),0) & -1< x_j < 0 ; \end{cases}
\]
we also send $z_j=1$ to $a_j$ and $z_j=-1$ to $b_j$, as suggested by
these formulae.
We then define $F$ as the map that sends $(z_j)_j\in\mathbb T^N$ to $(F(z_j))_j$.

Our preparatory discussion makes it clear that the induced map $J\mapsto z$ is continuous.
Since we clearly have a bijection between the compact metric spaces $\RR_0(K)$ and
$\mathbb T^N$, this map automatically has to be a homeomorphism.
\end{proof}
In this argument,
we established the continuity of the map $\RR_0(K)\to\mathbb T^N$ that we set
up above. As usual, this was sufficient because we have a bijection between compact metric spaces.
It is interesting to note that, contrary to what one would normally expect,
an easy direct argument for the continuity of the inverse map does not seem available.
Rather, we run into difficulties very similar to the ones that will occupy us in the main part of
this paper: Could it happen that for $z'\in\mathbb T^N$ arbitrarily close to a fixed
$z\in\mathbb T^N$, part (or all) of the measure $\chi_K\rho$ gets sucked out of $K$ and
moved into the gaps? The weak-$*$ continuity of $\rho$ does not prevent this, and if it happens,
then $(\sigma_j)_j$ could be discontinuous (so it does not happen here, but this an indirect
argument and in fact the one we just gave).
\section{Approximation by periodic operators}
Our main goal in this section is to prove Theorem \ref{T1.1}. Before we embark on this assignment,
we discuss how to derive Theorem \ref{T1.2} from Theorem \ref{T1.1}. We will need the following
observation. Actually, we will only need part (a) of Proposition \ref{P3.1} here, but part (b)
will be needed later on, in the proof of Theorem \ref{T1.4}, and this seems a good place to present it.
\begin{Proposition}
\label{P3.1}
Let $J_n,J\in\J_R$ and suppose that $d(J_n,J)\to 0$.\\
(a) If $J_n\in\RR(B_n)$ and $|B\setminus B_n|\to 0$, then $J\in\RR(B)$.\\
(b) If $h(K_n,K)\to 0$ and $\sigma(J_n)\subset K_n$, then $\sigma(J)\subset K$.

In particular, if $J_n\in\RR_0(K_n)$ and $\delta(K_n,K)\to 0$, then $J\in\RR_0(K)$.
\end{Proposition}
\begin{proof}
(a) Let $\epsilon>0$ be given. By passing to a subsequence, we can assume that
$|B\setminus B_n| < 2^{-n}\epsilon$. Let $A=\bigcap (B\cap B_n)$. Then
$J_n\in\RR(A)$, and $\RR(A)\cap\J_R$ is a compact set (see, for example \cite[Proposition 4.1(d)]{Remac}),
hence $J\in\RR(A)$ as well. Here, $|B\setminus A|<\epsilon$, and
$\epsilon>0$ was arbitrary, so from the definition of $\RR(B)$ we now
see that in fact $J\in\RR(B)$, as claimed.

(b) Let $x\in\R\setminus K$. Since $K$ is compact, we then have that $(x-2r,x+2r)\subset K^c$
for suitable $r>0$, and thus $(x-r,x+r)\subset K_n^c$ for all sufficiently large $n$. As
$J_n\to J$ in the strong operator topology, this implies that $x\notin\sigma(J)$;
see \cite[Theorem VIII.24(a)]{RS1}.
\end{proof}
Note that considerations of this type also give Proposition \ref{P1.1};
in fact, Proposition \ref{P1.1} could be viewed as a special case of the last part
of Proposition \ref{P3.1}, with $K_n=K$. Indeed, the Proposition says that
$\RR_0(K)$ is closed, so, since $\RR_0(K)\subset\J_R$ for large $R>0$ and
$\J_R$ is compact, it follows that $\RR_0(K)$ is compact also.

Now return to our original topic.
Assuming Theorem \ref{T1.1}, we obtain Theorem \ref{T1.2} as follows.
\begin{proof}[Proof of Theorem \ref{T1.2}]
Let $J\in\RR(B)$. Pick $J_n\in\RR_0(P_n)$ as in Theorem \ref{T1.1}. Periodicity means
that $S^{N_n}J_n=J_n$, where $N_n\in\N$ is the period,
and thus the $\varphi(J_n)$ are periodic as well.
Moreover, they have the same spectra as the $J_n$, and since a periodic operator is reflectionless
precisely on its spectrum, it follows that $\varphi(J_n)\in\RR(P_n)$. By continuity,
$\varphi(J_n)\to \varphi(J)$, and thus we can apply Proposition \ref{P3.1}(a)
to conclude the proof.
\end{proof}
\begin{proof}[Proof of Theorem \ref{T1.1}]
Let $J\in\RR(B)\cap\J_R$, with $B\subset (-R,R)$.
By removing a set of Lebesgue measure zero from $B$, we can
assume that $\rho_s(B)=0$; here, $\rho$ again denotes the measure that is associated with
the $H$ function of $J$,
as in \eqref{2.6}. In fact, by passing to a subset of almost the same Lebesgue measure,
we can also assume that $B$ is compact and has no isolated points.

Given the material from Section 2, especially Proposition \ref{P2.5}, we can describe
our goal as follows: We would like to construct finite gap sets $P_n$ (the periodicity
will follow from a density argument; this issue can be ignored for now) and
data $(\xi_n,\nu_{n,+})$ (say), corresponding to $J_n\in\RR_0(P_n)$, so that these
data approach the corresponding data $(\xi,\nu_+)$ of $J\in\RR(B)$, and $|B\Delta P_n|\to 0$.

Let $\epsilon>0$ be given.
The complement $B^c\cap (-R,R)$ is a disjoint union of bounded open intervals, and by keeping only
a finite large number of these gaps, we obtain a finite gap set
\[
A_0 = \bigcup_{j=1}^N J_j , \quad\quad J_j=[a_j,b_j]
\]
that satisfies $A_0\supset B$, $|A_0\setminus B|<\epsilon$, $\rho(A_0\setminus B)<\epsilon$.
Notice that the endpoints of these intervals $J_j$ belong to $B$. We chose $B$ so that $\rho_s(B)=0$.
Therefore, $\rho(\partial A_0)=0$. As a consequence, we obtain the following:
\begin{Lemma}
\label{L3.1}
If $\rho_n\to\rho$, then also $\chi_{A_0}\rho_n\to\chi_{A_0}\rho$ and $\chi_{A_0^c}\rho_n\to
\chi_{A_0^c}\rho$.
\end{Lemma}
As usual, these limits refer to the weak-$*$ topology.

We decompose
\[
\rho = \rho_1 + \rho_2 ,\quad\quad \rho_1=\chi_B\rho=\chi_B\rho_{ac} ,\quad \rho_2=\chi_{B^c}\rho ,
\]
and we write accordingly
\begin{equation}
\label{3.3}
\nu_+=\frac{1}{2}\rho_1 + f\rho_2 ,
\end{equation}
with a Borel function $f$ that satisfies $0\le f\le 1$. We know that
this density equals $1/2$ on $B$ because $J\in\RR(B)$; compare the discussion
that precedes Proposition \ref{P2.2}.
For future use, we record that
\begin{equation}
\|\chi_{A_0}\rho-\rho_1\| = \|\chi_{A_0^c}\rho-\rho_2\|  =\rho(A_0\setminus B) < \epsilon .
\label{3.2b}
\end{equation}

We now construct preliminary versions $A_n$ of the sets $P_n$
and associated Krein function $\xi_n$ (which will
correspond to certain $J_n\in\RR_0(A_n)$, to be specified at a later stage).
For each component
$I\subset B^c\cap (-R,R)$, we have either $I\subset A_0^c$ or $I\subset J_j\subset A_0$.
Let $\{ I_n \}_{n\ge 1}$ be an enumeration of those gaps of $B$
that are contained in $A_0$, and set
\[
\widetilde{A}_n = A_0 \setminus \bigcup_{j=1}^n I_j .
\]
This is still a finite gap set, and $B\subset\widetilde{A}_n\subset A_0$.
Obtain $A_n$ from $\widetilde{A}_n$ by subdividing each component of
$\widetilde{A}_n^c\cap (-R,R)$ into $n$ smaller gaps of equal length,
which are separated by $n-1$ very small intervals
of size $\delta_n>0$ each; these intervals will be referred to as \textit{bands }from now on;
the band size $\delta_n$ will be chosen later.

On each of these new (small) gaps, let $\xi_n$ jump
from $0$ to $1$ in a such a way that the average has the correct value.
More precisely, if $(a,b)$ is a gap of $A_n$, then put $\xi_n=\chi_{(\mu,b)}$ on $(a,b)$, where
$\mu=b-\int_a^b\xi\, dx$ (so $\int_a^b\xi_n=\int_a^b\xi$). If this leads to $\mu=b$ and $b$
is the left endpoint of a band, then we delete this band and set $\xi_n=0$ on the slightly
larger interval $(a,b+\delta_n)$. Do this for all such gaps (if any). Then use
the analogous procedure for those gaps where $\mu=a$ and $a$ is the
right endpoint of a band. Finally, we set $\xi_n=1/2$ on $A_n$.

This Krein function $\xi_n$ lies in $X(A_n)$; it corresponds to certain operators
from $\RR_0(A_n)$. As usual, it also defines a measure $\rho_n$ via \eqref{2.5}, \eqref{2.6}.
We claim that if the $\delta_n$ approach zero sufficiently rapidly, then we can achieve that
\begin{equation}
\label{3.1}
\rho_n(A_n\setminus A_0)\to 0 .
\end{equation}
This will certainly follow if we can show that $\rho_n$ gives little weight to the newly
introduced bands of $A_n$. We need the following auxiliary calculation.
\begin{Lemma}
\label{L3.2}
Fix $0<A,B<R$ and define, for sufficiently small $\delta>0$,
\[
\xi_{\delta}(x) = \begin{cases} 0 & \delta<x<B \\ 1 & -A<x<0 \\ 1/2 & 0 < x<\delta
\end{cases}.
\]
Then
\[
\lim_{\delta\to 0+} \sup \rho([0,\delta]) = 0 ,
\]
where the supremum is taken over all $\rho$ whose Krein functions $\xi\in L^{\infty}(-R,R)$
agree with $\xi_{\delta}$ on $(-A,B)$.
\end{Lemma}
\begin{proof}
Since $\xi=1/2$ on $(0,\delta)$, it follows that $\rho$ is purely absolutely continuous
on this interval, with density equal to $(1/\pi)|H(x)|$. From \eqref{2.5}, we have that
for $0<x<\delta$,
\[
|H(x)| = (x+R) h_{\delta}(x) \exp\left( \int_{(-R,R)\setminus (-A,B)} \frac{\xi(t)\, dt}{t-x} \right) ,
\]
where $h_{\delta}(x) = \lim_{y\to 0+}
\exp\left( \int_{-A}^B \frac{(t-x)\xi_{\delta}(t)\, dt}{(t-x)^2+y^2}\right)$.
In fact, this can be evaluated:
\[
h_{\delta}(x) = \frac{\sqrt{x(\delta -x)}}{A+x}
\]
These formulae make it clear, first of all,
that the arrangement that maximizes $\rho([0,\delta])$ is the one
where $\xi=0$ on $(-R,-A)$ and $\xi=1$ on $(B,R)$. It is then straightforward to estimate
$\rho([0,\delta])$ for
this measure and confirm that this quantity approaches zero as $\delta\to 0+$.
We leave the details to the reader.
\end{proof}
This Lemma indeed establishes \eqref{3.1} because we deleted those bands for which we don't have
the situation described in the Lemma ($\xi_n=1$ to the left of the band and $\xi_n=0$ to the right).

Moreover, by taking $\delta_n\to 0$ so fast that also $|A_n\setminus B|\to 0$, we can make sure that
$\xi_n\, dx \to \xi\, dx$ in weak-$*$ sense. This follows now because $\xi=1/2$ on $B$,
so if a continuous $f$ is given, we can split
\begin{equation}
\label{3.71}
\int f\xi_n\, dx = \frac{1}{2} \int_{A_n} f\, dx + \int_{A_n^c} f\xi_n\, dx .
\end{equation}
The first integral on the right-hand side converges to $\int_B f\xi\, dx$ as $n\to\infty$.
To deal with the last integral from \eqref{3.71}, we approximate $f$ on $B^c\supset A_n^c$
uniformly by functions $g_n$ that are
constant on the gaps of $A_n$, and now our definition of $\xi_n$ on these intervals guarantees that
$\int_I g_n\xi_n\, dx = \int_I g_n\xi\, dx$ for each gap $I$ of $A_n^c$. Since
$|B^c\setminus A_n^c|\to 0$, this implies that $\int_{A_n^c} f\xi_n\to \int_{B^c} f\xi$.

It follows from this that $\rho_n\to\rho$. Indeed, the weak-$*$ convergence $\xi_n\, dx\to\xi\, dx$
clearly implies that $H_n(z)\to H(z)$ locally uniformly on $\C^+$, and this in turn shows that
$\rho_n\to\rho$; compare again \cite[Theorem 2.1]{Remac}.

Thus we also have that $\chi_{A_0}\rho_n\to\chi_{A_0}\rho$
and $\chi_{A_0^c}\rho_n\to\chi_{A_0^c}\rho$, by Lemma \ref{L3.1}.

Next, pick a $g\in C_0^{\infty}(\R)$ so that $0\le g\le 1$ and
\begin{equation}
\label{3.41}
\|f\rho_2-g\rho_2\| < \epsilon ,
\end{equation}
where $f$ is the density from \eqref{3.3}.
Since $g$ is continuous, we then have that
$g\chi_{A_0^c}\rho_n\to g\chi_{A_0^c}\rho$, and, by \eqref{3.1},
it is also true that
\[
g\chi_{A_0^c\cap A_n^c}\rho_n\to g\chi_{A_0^c}\rho .
\]
These approximating measures are pure point measures:
\begin{equation}
\label{3.101}
\chi_{A_0^c\cap A_n^c}\rho_n = \sum_{j=1}^{N_n} w_j^{(n)}\delta_{x_j^{(n)}}
\end{equation}
Our original goal was to approximate $\nu_+$ from \eqref{3.3} by a $\nu_{n,+}$ that is
obtained by splitting $\rho_n$ and thus corresponds
to a $J_n\in\RR_0(A_n)$. Recall from Section 2 how those plus measures were obtained: we split
$\rho=\nu_+ + \nu_-$, and if $J\in\RR(B)$, then we have to put exactly one half of the absolutely
continuous part of $\rho$ on $B$ into $\nu_+$. For a general $J\in\RR(B)$, there is no restriction on
how to distribute the singular part between $\nu_+$ and $\nu_-$. However, our goal is
to construct approximations $J_n\in\RR_0(A_n)$ (note the index $0$!),
and then either all or nothing of each
point mass $w_j^{(n)}\delta_{x_j^{(n)}}$ has to go into $\nu_{n,+}$. This is an unwelcome
restriction because ideally we would have liked to put the fraction $g(x_j^{(n)})w_j^{(n)}
\delta_{x_j^{(n)}}$ into $\nu_{n,+}$.
We overcome this obstacle by splitting each point mass into two new point masses whose ratio is at
our disposal.
\begin{Lemma}[The Splitting Lemma]
\label{L3.3}
Let $0<A,B<R$, and let $\xi_0:(-R,R)\to [0,1]$ be a Borel function whose restriction to
$(-A,B)$ is $\chi_{(0,B)}$. For sufficiently small
$\delta>0$ and fixed $0 < g < 1$, define
\[
\xi_{\delta}(x) = \begin{cases} 1 & -g\delta < x < 0 \\ 1/2 & 0<x<\delta^2 \\
0 & \delta^2<x<(1-g)\delta \\
\xi_0(x) & \textrm{\rm otherwise} \end{cases} .
\]
Let $\rho_0,\rho_{\delta}$ be the measures that are associated with
$\xi_0$ and $\xi_{\delta}$, respectively.
Then, as $\delta\to 0+$, we have that $\rho_{\delta}\to\rho_0$,
$\rho_{\delta}([0,\delta^2])\to 0$, and
\[
\rho_{\delta}(\{-g\delta\} ) \to g \rho_0(\{ 0\}), \quad\quad\quad
\rho_{\delta}(\{ (1-g)\delta \})\to (1-g)\rho_0(\{ 0\}) ,\\
\]
\end{Lemma}
So a point mass inside a gap (at $x=0$ here) can be split into two nearby point masses of
approximately the same total weight, with a ratio between the two that can be specified in advance,
by introducing an additional tiny band at the original point mass.
\begin{proof}
This is proved by an explicit calculation. First of all, we immediately obtain the
weak-$*$ convergence $\rho_{\delta}\to\rho_0$ from $\|\xi_{\delta}-\xi_0\|_1\to 0$.
Next, we can control $\rho_{\delta}([0,\delta^2])$ by a calculation similar to the one that was
used in the proof of Lemma \ref{L3.2}, so we will explicitly discuss only the point masses of
$\rho_{\delta}$ here. Observe also in this context that $\rho_{\delta}$ on $(-A,B)$ is supported
by $\{-g\delta\}\cup\{(1-g)\delta\}\cup (0,\delta^2)$.

We will use the formulae
\begin{align*}
\rho_{\delta}(\{-g\delta\}) & = \lim_{y\to 0+} y |H_{\delta}(-g\delta+iy)| ,\\
\rho_0(\{0\}) & = \lim_{y\to 0+} y |H_0(iy)| ,
\end{align*}
where, as usual, $H_{\delta},H_0$ are the $H$ functions of $\xi_{\delta}$ and $\xi_0$,
respectively, as in \eqref{2.5}. We rewrite \eqref{2.5} as
\begin{align*}
H_{\delta}(z) & = h(z) \exp \left( \int_{-A}^B \frac{\xi_{\delta}(t)\, dt}{t-z} \right) , \\
H_0(z) & = h(z) \exp \left( \int_{-A}^B \frac{\xi_0(t)\, dt}{t-z} \right) ;
\end{align*}
here, $h(z)$ is independent of $\delta$ and holomorphic in a neighborhood of $(-A,B)$.
Explicit calculation now shows that $\rho_0(\{ 0\})=Bh(0)$ and
\[
\rho_{\delta}(\{-g\delta\}) = g^{1/2} (g-\delta)^{1/2}(B+g\delta)h(-g\delta) ,
\]
which obviously converges to $gBh(0)$ as $\delta\to 0+$.

The calculation for $\rho_{\delta}(\{(1-g)\delta\})$ is of course analogous; alternatively,
we could combine the previous calculations with the fact that $\rho_{\delta}\to\rho_0$.
\end{proof}
Apply the Splitting Lemma to all point masses from \eqref{3.101},
with $x_j^{(n)}$ taking the role of $x=0$ in the Lemma,
and with $g=g_j^{(n)}=g(x_j^{(n)})$
and the $\delta$'s chosen so small that the statements below will be true.
More precisely, we only do this if $g_j^{(n)}\not= 0,1$; there is of course no need
to split the point mass if we already have $g_j^{(n)}=0$ or $1$.

For reasons that will become clear in a moment, we will also apply the Splitting Lemma
to the point masses of $\rho_n$ on $A_0\cap A_n^c$, with $g=1/2$.
We obtain new sets from the $A_n$'s, with additional, very small bands added.
We call these news sets $P_n$, and otherwise use tildes to refer to the new, modified data.
Again, we have that
\[
\chi_{A_0^c\cap P_n^c} \widetilde{\rho}_n = \sum_{j=1}^{\widetilde{N}_n}
\widetilde{w}_j^{(n)}\delta_{\widetilde{x}_j^{(n)}} ;
\]
the point is that we can now achieve that
\begin{equation}
\label{3.4}
\sum_{\widetilde{x}_j^{(n)}\notin A_0} \widetilde{\sigma}_j^{(n)}
\widetilde{w}_j^{(n)}\delta_{\widetilde{x}_j^{(n)}}\to g\chi_{A_0^c}\rho
\end{equation}
for suitably chosen $\widetilde{\sigma}_j^{(n)}\in\{ 0,1\}$. More precisely, for each fixed set
of indices $j,n$, we split
the old point mass at $x_j^{(n)}$ into two new point masses. One of these has weight approximately
equal to $\widetilde{w}=g(x_j^{(n)})w_j^{(n)}$, and we set
the corresponding $\widetilde{\sigma}=1$, and $\widetilde{\sigma}=0$ for the other new point mass of
this pair. This procedure makes sure that the point masses from \eqref{3.4} with
$\widetilde{\sigma}=1$ can be put
in one-to-one correspondence with those of $\chi_{A_0^c\cap A_n^c}\rho_n$, and the two corresponding
point masses $x_j^{(n)}$, $\widetilde{x}_j^{(n)}$
will get arbitrarily close to each other and the weights will satisfy
$\widetilde{w}\approx g(x)w$, up to an error that will
approach zero as we take the $\delta$'s from the Splitting Lemma closer and closer to zero.
So we can indeed make sure that \eqref{3.4} holds.

Similarly and as already announced above,
the Splitting Lemma with $g=1/2$ can be used on $A_0\setminus A_n$, and we then
obtain that (for suitable $\widetilde{\sigma}$; as before we can take one $\widetilde{\sigma}=1$
and the other equal to $0$ for each split pair)
\begin{equation}
\label{3.21}
\frac{1}{2}\widetilde{\rho}_{n,ac} + \sum_{\widetilde{x}_j^{(n)}\in A_0}
\widetilde{\sigma}_j^{(n)}\widetilde{w}_j^{(n)}\delta_{\widetilde{x}_j^{(n)}} \to
\frac{1}{2}\chi_{A_0}\rho .
\end{equation}
This follows because $\chi_{A_0}\rho_n\to\chi_{A_0}\rho$, as we saw above, and the
left-hand side of \eqref{3.21} is close to $(1/2)\chi_{A_0}\rho_n$. Let us explain
in more detail why this is true. In fact, it will only be true if the band sizes
$\delta$ are chosen small enough when we apply Lemma \ref{L3.2} and the Splitting Lemma.

Notice that while $\widetilde{\rho}_{n,ac}$ is supported
by $P_n$, only very little weight is given to $\widetilde{A}_n^c$ if the band sizes
$\delta$ were chosen small enough. The set $\widetilde{A}_n$ was introduced at the
beginning of this proof; it is equal to $P_n$, but with the bands removed.
So $\widetilde{\rho}_{n,ac}$ is close (in fact, in norm)
to its restriction to $\widetilde{A}_n$, and this in turn is close (again, in norm)
to $\chi_{A_n\cap A_0}\rho_n$ (no tilde!), because the main part of this measure also sits
on $\widetilde{A}_n$, which is
at some distance from the small set where we changed $\xi$ when going from
$\rho_n$ to $\widetilde{\rho}_n$. The part of $(1/2)\rho_n$ on $A_0\cap A_n^c$, on the other hand,
is approximated by the sum from \eqref{3.21}, by its construction.

By combining \eqref{3.21} with \eqref{3.4}, we obtain that
\[
\frac{1}{2}\widetilde{\rho}_{n,ac} + \sum
\widetilde{\sigma}_j^{(n)}\widetilde{w}_j^{(n)}\delta_{\widetilde{x}_j^{(n)}} \to
\frac{1}{2}\chi_{A_0}\rho + g\chi_{A_0^c}\rho .
\]
If we call the measure on the right-hand side $\widetilde{\nu}_+$,
then \eqref{3.3}, \eqref{3.2b}, \eqref{3.41} show that $\|\widetilde{\nu}_+-\nu_+\|<(5/2)\epsilon$.
Moreover, the measure on the left-hand side is a measure of the type
$\widetilde{\nu}_{n,+}$; it corresponds to some (unique) $J_n\in\RR_0(P_n)$.
We can also make sure that $|P_n\Delta B|<\epsilon$ here. Indeed, $P_n\supset B$,
and from the way $P_n$ was constructed, it is clear that we can make $|P_n\setminus B|$
arbitrarily small. Here it becomes again essential to add only very small bands during
the construction. Recall also that we perhaps replaced $B$ with a slightly smaller set
at the very beginning of the proof, so the originally given set $B$ is not necessarily
a subset of $P_n$, and we can really only make a claim about
$P_n\Delta B$.

By the material from Section 2, especially Propositions \ref{P2.5}, \ref{P2.2}, we know that the
Jacobi matrices $J_n\in\RR_0(P_n)$ that correspond to the spectral data we have constructed
will come as close to $J\in\RR(B)$ as we wish, provided $\epsilon>0$ was taken sufficiently
small and $n$ is large. This follows because, by construction, the spectral data
$(\xi_n,\nu_{n,+})$ of $J_n$ will come arbitrarily close to those of $J$.
At the same time, $|P_n\Delta B|$ can also be made arbitrarily small.

These operators $J_n\in\RR_0(P_n)$ are not necessarily periodic. However,
it is known \cite{Bogat,Pehers,Totik}
that an arbitrary finite gap set can
be transformed into a periodic set by arbitrarily small perturbations of the endpoints of its
intervals. Here, we call a finite gap set $P$ periodic if every (equivalently: one)
$J\in\RR_0(P)$ is periodic. Therefore, a final small adjustment of our sets $P_n$ will give the
full claim. We then need to know that for the new sets $\widetilde{P}_n$ we will
be able to find new Jacobi matrices $\widetilde{J}_n\in\RR_0(\widetilde{P}_n)$ that are close
to the original operators $J_n\in\RR_0(P_n)$. This issue will be discussed in great
detail later in this paper; here we only need a small part of these later
results. We can refer to
Theorem \ref{T4.6} and the discussion preceding it
to finish the present proof. This result applies here because for
a finite gap set $P$, it is definitely true that $\rho_s(P)=0$ for all $\rho\in\mathcal H(P)$.
\end{proof}
\section{The distance $\delta$}
In this section, we prove Theorem \ref{T1.4}(a). This will be an easier discussion
than the proof of part (b) of this Theorem and will thus serve as a good warm-up.
As a preliminary, we first confirm that $\delta$ is a metric.
\begin{Proposition}
\label{P5.1}
\eqref{defdelta} defines a metric on the non-empty compact subsets of $\R$.
\end{Proposition}
\begin{proof}
It is of course well known that $h$ is a metric, and clearly $\delta$ is symmetric
and non-negative, so it suffices to show that $|K\Delta L|$ satisfies the triangle inequality.
This, however, follows immediately from the observation that
\[
A\Delta C \subset (A\Delta B) \cup (B\Delta C)
\]
for any three sets $A,B,C$.
\end{proof}
It will also be useful to keep in mind the following basic fact about Hausdorff
distance: If $A_n,A$ are non-empty, compact subsets of a compact metric
space $X$ and $h(A_n,A)\to 0$, then
\[
A = \{ x\in X: a_j\to x \textrm{ for some sequence }a_j\in A_{N(j)}, N(j)\to\infty \} ,
\]
and also
\[
A = \{ x\in X: a_n\to x \textrm{ for some sequence }a_n\in A_n \} .
\]
Compare, for example, \cite[Lemma 1.11.2]{vanMill}.
\begin{proof}[Proof of Theorem \ref{T1.4}(a)]
We will actually show that if we just know that the collections of
$\xi$ functions $X(K),X(K')$ are close in Hausdorff distance, that already forces
$\delta(K,K')$ to be small as well. So we can completely avoid all issues related to the
splitting $\rho=\nu_++\nu_-$ of the measures $\rho$.

Let $K_n,K$ be non-empty compact subsets of $[-R,R]$, so that $h(\Omega_n,\Omega)\to 0$.
Suppose, first of all, that it were not true that $h(K_n,K)\to 0$, say $h(K_n,K)\ge 2r >0$
(on a subsequence, which, for simplicity, we assume to be the original sequence).
We pass to another subsequence if necessary and then find ourselves in one of the
following two situations. Either: (i) There are $x_n\in K_n$ so that $(x_n-2r,x_n+2r)\cap
K = \emptyset$; or (ii) There are $x_n\in K$ so that $(x_n-2r,x_n+2r)\cap K_n = \emptyset$.
In both cases, we can make the $x_n$ converge to $x\in\R$ by passing to still another subsequence.

Case (i) can then be ruled out as follows: Let $\xi_n=1$ on all gaps (of $K_n$)
to the left of $x_n$ and $\xi_n=0$ on all gaps to the right of $x_n$, and, of course,
$\xi_n=1/2$ on $K_n$. Then $\xi_n\in X(K_n)$ and thus there are $J_n\in\Omega_n$
that have these functions as their $\xi$ functions,
but we claim that no accumulation point of the $\xi_n$ lies in $X(K)$.
This is a contradiction because we can pass to a convergent subsequence so that
$J_n\to J$. Since $h(\Omega_n,\Omega)\to 0$, the limit must satisfy $J\in\Omega$.
By Proposition \ref{P2.5}, this implies that $\xi_n\, dt\to\xi\, dt$,
where $\xi\in X(K)$ is the $\xi$ function of $J$.

Since $(x-r,x+r)\cap K=\emptyset$, an arbitrary $\xi\in X(K)$
can only take the values $0$ and $1$ on this interval, and if both values occur,
then $\xi$ has to jump from $0$ to $1$ at some point $\mu\in [x-r,x+r]$.
Suppose that we had $\xi_{n_j}\, dt \to \xi\, dt$ in weak-$*$ sense for such a $\xi$
and a subsequence of the sequence $\xi_n$ that was defined above. Now $\xi_n\ge 1/2$
on $(x-r,x-\delta)$ for all large $n$ for arbitrary $\delta>0$, so by testing against
functions $f\in C(\R)$ that are supported by $[x-r,x]$, equal to $1$ on some subinterval
and take values between $0$ and $1$, we see that we must have $\mu=x-r$, that is,
$\xi=1$ on $[x-r,x+r]$. This, however, leads to a contradiction when we test against
similar functions that are supported by $[x,x+r]$.
We have to admit that the accumulation points of the sequence $\xi_n$ are not $\xi$ functions
of operators $J\in\Omega$, but, as explained above,
this contradicts our hypothesis that $h(\Omega_n,\Omega)\to 0$.

Case (ii) is handled similarly. As observed earlier, we can in fact assume that there is
an $x\in K$ so that $(x-r,x+r)\cap K_n=\emptyset$ for all large $n$.
This time, we use the procedure from above
to define a $\xi\in X(K)$: Put
$\xi=1$ on all gaps of $K$ to the left of $x$, $\xi=0$ on the gaps to the right of $x$,
and $\xi=1/2$ on $K$. We can now argue as above to show that this $\xi$ can not be the limit of
a (sub-)sequence of $\xi$ functions $\xi_n\in X(K_n)$, basically, as before, because
these functions can only jump from $0$ to $1$ on $(x-r,x+r)$, not from $1$ to $0$.
Again, it would follow that $h(\Omega_n,\Omega)\not\to 0$, which is a contradiction.
So we have now shown that $h(K_n,K)\to 0$. Note that this in particular implies that
$\min K_n\to\min K$, $\max K_n\to\max K$.

It remains to show that $|K_n\Delta K|\to 0$ as well.
We first observe that we must definitely have
that $|K_n|\to |K|$. Indeed, if this were false, say $|K_n|\le |K|-\epsilon$ on a subsequence,
then the sequence of functions $\xi_n\in X(K_n)$ that are equal to $0$ on all gaps
cannot have a limit point
$\xi\in X(K)$. If, on the other hand, we had that $|K_n|\ge |K| + \epsilon$,
then the function $\xi\in X(K)$ that is equal to $0$ on all gaps cannot be reached as a limit
of any sequence $\xi_n\in X(K_n)$, so again we obtain a contradiction.

This argument also shows that, more generally,
\[
|K_n\cap I|\to |K\cap I|
\]
for every fixed interval $I\subset\R$. To establish this version, just argue as above, but test against
functions that are close to $\chi_I$.

In particular, $I$ can be any gap of $K$ here, and it then follows that
$|K_n\cap I|\to 0$. Since for any $\epsilon>0$, we can find finitely many gaps
so that the total measure of the remaining gaps is $<\epsilon$, this implies that
\[
\lim_{n\to\infty} \left| K_n\setminus K\right| = 0 .
\]
Since
\[
\left| K\setminus K_n\right| = \left| K\right|-\left| K_n\right|+\left|
K_n\setminus K \right|
\]
and, as pointed out at the beginning of this argument, $|K_n|\to |K|$,
we also obtain that $|K\setminus K_n|\to 0$.
\end{proof}
\section{The spaces $\RR_0(K)$}
We will first discuss how Theorem \ref{T1.4}(b) follows from its variant Theorem \ref{T1.5}
and then prove this statement.
Suppose that what Theorem \ref{T1.4}(b) asserts were not true.
Then there are compact sets $K_n,K$ with $\delta(K_n,K)\to 0$, but $h(\Omega_n,\Omega)\ge\epsilon>0$.
So, on a subsequence (which, for notational simplicity, we assume to be the original sequence),
one of the following alternatives will hold: (i) There are $J_n\in\Omega_n$ so that
$d(J_n,\Omega)\ge\epsilon$; (ii) There are $J_n\in\Omega$ so that $d(J_n,\Omega_n)\ge\epsilon$.
By compactness, the $J_n$ will approach a limit $J$ on a subsequence. In case (i), Proposition
\ref{P3.1} forces $J\in\Omega$, an obvious contradiction. This rules out case (i),
and we have also inadvertently established Theorem \ref{T1.5}(a).

So it just remains to discuss case (ii), and we can actually restrict our attention to the
slightly simpler scenario:\\[0.2cm]
\textit{(ii') There exists $J\in\Omega$ so that}
\begin{equation}
\label{4.86}
\limsup_{n\to\infty} d(J,\Omega_n) > 0 .
\end{equation}
Theorem \ref{T1.5}(b) claims that this can not happen if $\rho_{sc}(K)=0$, where,
as usual, $\rho$ denotes the measure of the $H$ function of $J$, as in \eqref{2.6}.
So Theorem \ref{T1.4}(b) will follow if we can prove Theorem \ref{T1.5}(b).

We begin our discussion with the special case when $\rho_s(K)=0$.
\begin{Theorem}
\label{T4.6}
Suppose that $\delta(K_n,K)\to 0$, and suppose further that
$J\in\Omega$ with $\rho_s(K)=0$. Then \eqref{4.86} cannot hold: $d(J,\Omega_n)\to 0$.
\end{Theorem}
\begin{proof}
We know that
\begin{align}
\label{4.33}
\rho & = \rho_{ac} + \sum w_j\delta_{\mu_j}, \\
\label{4.33a}
\nu_+ & = \frac{1}{2}\, \rho_{ac} + \sum \sigma_j w_j\delta_{\mu_j} ,
\end{align}
and our goal is to find $J_n\in\Omega_n$ so that $\xi_n\, dx\to\xi\, dx$ and $\nu_{n,+}\to\nu_+$ in
the weak-$*$ topology. The sums in \eqref{4.33}, \eqref{4.33a} are taken over those
gaps $(a_j,b_j)$ for which $\mu_j\not=a_j, b_j$; in other words, these sums give us
the singular parts of $\rho$, $\nu_+$ on $K^c$ (which is the complete singular part
here, by assumption).

We start by labeling the gaps of $K$ once and for all by integers $j\ge 1$, for example in order
of decreasing size. For each gap
$(a_j,b_j)\subset K^c$ of $K$ and sufficiently large $n$,
there has to be a corresponding gap $(a_j^{(n)},b_j^{(n)})\subset K^c_n$ that converges to $(a_j,b_j)$
in the sense that $a_j^{(n)}\to a_j$, $b_j^{(n)}\to b_j$. We use this fact to introduce integers
$N(j)$, for $j\ge 1$, as follows: If $\mu_j\not= a_j,b_j$, then we define $N(j)$ as the smallest index
for which every $K_n$ for $n\ge N(j)$
has a gap $(A,B)$ that satisfies
\[
\left| A - a_j \right| + \left| B-b_j\right| < \frac{1}{10}\, \min\{b_j-\mu_j, \mu_j-a_j\} ,
\]
say. If $\mu_j=a_j$ or $b_j$, we proceed similarly, but replace the minimum with $b_j-a_j$ here.

Note that if $n\ge N(j)$, then the gap $(A,B)$ of $K_n$ that is close
to $(a_j,b_j)\subset K^c$ in this sense is unique,
so we can label the gaps of $K_n$ so that this gap $(A,B)$ (where $A,B$ depend on $n$) also
gets the label $j$, for $n\ge N(j)$. We do not impose any conditions on the labels of the
remaining gaps of $K_n$, other than the obvious requirement that no label can be used more than once
for fixed $n$. It could actually happen here that when applying this procedure,  we run out of
labels $j\ge 1$. More precisely, this happens if there is a finite set $M\subset\N$
so that $N(j)$ is bounded on $\N\setminus M$, but $K_n$ for $n\ge \sup_{j\notin M} N(j)$ has
more than $|M|$ additional gaps.
In this case, we just invent new labels; for example, we could use negative integers.

Each $\mu_j$ from \eqref{4.33} is in the interior
of its gap, and we can put $\mu_j^{(n)}=\mu_j$, at least for $n\ge N(j)$.
If $n\ge N(j)$ and $\mu_j=a_j$, then we put $\mu_j^{(n)}=a_j^{(n)}$; the analogous
procedure is used if $\mu_j=b_j$ and $n\ge N(j)$. Finally,
on the remaining gaps of $K_n$, we can put $\xi_n=0$, say.

The assumption that $\delta (K_n,K)\to 0$ then makes sure that $\xi_n$ differs
from $\xi$ only on a small set. In particular, $\|\xi_n-\xi\|_1\to 0$, and, as usual, this implies
that $\rho_n\to\rho$ in weak-$*$ sense. We rewrite \eqref{4.33a} as
\[
\nu_+ = \frac{1}{2}\rho + \sum \left( \sigma_j-\frac{1}{2} \right) w_j\delta_{\mu_j} ,
\]
and our goal is to find parameters $\sigma_j^{(n)}$, $g_n(x)$ so that the corresponding measures
\begin{equation}
\label{4.36}
\nu_{n,+} = \frac{1}{2}\rho_n + \left( g_n - \frac{1}{2} \right)\chi_{K_n} \rho_{n,s} +
\sum\left( \sigma_j^{(n)}-\frac{1}{2} \right) w_j^{(n)}\delta_{\mu_j^{(n)}}
\end{equation}
approach $\nu_+$. As will become clear from the argument we are about to give,
the existence of a singular part of $\rho_n$ on $K_n$ would only make this task easier
because the functions $g_n$, $0\le g_n\le 1$, are completely at our disposal.
In fact, we could just set $g_n(x)\equiv 1/2$ here right away and then proceed
as outlined below.
We will therefore only discuss the case where $\rho_{n,s}(K_n)=0$ explicitly.

We will establish the desired weak-$*$ convergence $\nu_{n,+}\to\nu_+$
as follows: We will verify that
given $\epsilon>0$ and
finitely many intervals $I_1,\ldots, I_N\subset\R$ with $\rho(\partial I_j)=0$,
there exists $n_0\in\N$ so that for all $n\ge n_0$, it is possible to
assign values to the $\sigma_j^{(n)}$ so that
\begin{equation}
\label{4.38}
\left| \nu_{n,+}(I_j)-\nu_+(I_j)\right| < \epsilon .
\end{equation}
We can assume that the
$I_j$ are disjoint, and we can then focus on a single interval $I$ because
$\nu_{n,+}(I)$ obviously only depends on those $\sigma_j^{(n)}$ for which $\mu_j^{(n)}\in I$.

So fix such an interval $I$. We make the obvious first step in our attempts to choose the
parameters $\sigma_j^{(n)}$ appropriately: we put $\sigma_j^{(n)}=\sigma_j$ if $n\ge N(j)$.
To find suitable values for the remaining $\sigma$'s, we make the following preliminary
observation: If $N_n\in\N$, $N_n\to\infty$, then
\begin{equation}
\label{4.37}
\lim_{n\to\infty} \sup_{j\notin\{ 1,\ldots, N_n \} } w_j^{(n)} = 0 .
\end{equation}
To prove this, we argue by contradiction. Suppose \eqref{4.37} were wrong. Then there are
$\mu_{j_n}^{(n)}$ with arbitrarily large $n$ so that (on this subsequence, which we won't make
explicit in the notation, as always) eventually $j_n\notin \{ 1,\ldots, N \}$ for every $N\ge 1$ and
\[
w_{j_n}^{(n)}=\rho_n(\{ \mu_{j_n}^{(n)} \} ) \ge \epsilon > 0 .
\]
Here, we may also assume that $\mu_{j_n}^{(n)}\to x$. But then the weak-$*$ convergence
$\rho_n\to\rho$ implies that $\rho(\{ x\} )>0$. Since $\rho_s(K)=0$ by hypothesis, this forces
$x$ to lie in some gap $(a_j,b_j)\subset K^c$, but then only $\mu_j^{(n)}$ can be close to
$x$ for large $n$, so in particular it is not possible to have indices $j_n\notin\{ 1,\ldots, j\}$.

We now split $\nu_{n,+} = \nu_n^{(1)} + \nu_n^{(2)}$, where
\[
\nu_n^{(2)} = \sum_{j\notin \{ 1,\ldots, N_n\} }
\left( \sigma_j^{(n)}-\frac{1}{2} \right) w_j^{(n)}\delta_{\mu_j^{(n)}} .
\]
Here, we take cut-offs $N_n\in\N_0$ that satisfy $N_n\to\infty$ but increase so slowly that
$N(j)\le n$ if $1\le j\le N_n$, and
\[
\lim_{n\to\infty} \sum_{\substack{j=1,\ldots,N_n\\\mu_j\not= a_j, b_j}}
\left| w_j^{(n)} - w_j \right| = 0 .
\]
This is possible because $w_j^{(n)}\to w_j$ for fixed $j$, and this latter statement
just follows from the convergence $\rho_n\to\rho$ together with the fact that if
$\mu_j\not= a_j, b_j$, then
$\rho_n$ is supported by $\{\mu_j \}$ in a neighborhood of $\mu_j=\mu_j^{(n)}$ for all large $n$.

Notice that
\[
2\left| \nu_n^{(1)}(I) - \nu_+(I) \right| \le \left| \rho_n(I)-\rho(I) \right|
+ \sum_{\substack{j=1,\ldots,N_n\\\mu_j\not= a_j, b_j}}
\left| w_j^{(n)} - w_j \right| + \sum_{j>N_n} w_j ,
\]
and the right-hand side approaches zero here.
So now our task is to show that
for large $n$, the remaining $\sigma_j^{(n)}$ can be chosen so that $|\nu_n^{(2)}(I)|$ becomes small.
This, however, follows immediately from \eqref{4.37}: We have $\nu_n^{(2)}(I)\ge 0$ if we take
all the $\sigma$'s equal to $1$ and $\nu_n^{(2)}(I)\le 0$ if we set them all equal to $0$, and
\eqref{4.37} says that we can go from one extreme value to the other in very small steps
by changing individual $\sigma$'s, so we will be able to come close to zero
and thus obtain \eqref{4.38} for all large $n$.
\end{proof}

Given this, Theorem \ref{T4.1} will now indeed imply Theorem \ref{T1.5}(b) because if $J\in\RR_0(K)$
with $\rho_{sc}(K)=0$ and $\epsilon>0$ are given, then
we can first use Theorem \ref{T4.1} to find a $J'\in\RR_0(K)$ with $\rho'_s(K)=0$ and $d(J,J')<\epsilon$
and then Theorem \ref{T4.6} says that $d(J',\Omega_n)\to 0$,
so $\limsup d(J,\Omega_n)< \epsilon$.
\begin{proof}[Proof of Theorem \ref{T4.1}]
Our assumption says that
\begin{align}
\label{4.9a}
\rho & = \rho_{ac} + \sum v_j \delta_{x_j} + \sum w_j\delta_{\mu_j}, \\
\label{4.9b}
\nu_+ & = \frac{1}{2}\,\rho_{ac} + \sum g_jv_j\delta_{x_j} + \sum \sigma_j w_j\delta_{\mu_j} ,
\end{align}
where $x_j\in K$, $\mu_j\in (a_j,b_j)\subset K^c$, and $0\le g_j\le 1$, $\sigma_j=0, 1$.
We will first show that given $\epsilon>0$, we can find $J'\in\RR_0(K)$ so that
$d(J,J')<\epsilon$ and $\rho'_s(K)<\epsilon$. This will be done by removing
sufficiently many of the point masses $\delta_{x_j}$ by modifying
$\xi$ on small neighborhoods of the $x_j$.

To keep the argument transparent and for notational convenience, we will start with
the special case where we remove just one point mass, say $v_1\delta_{x_1}$. We also
assume that $x_1=0$.

We will make use of Proposition \ref{P2.4}. So we will try to construct new
data $\xi',\nu'_+$ so that $\nu'_+$ is close to $\nu_+$ in the weak-$*$ topology,
and $\rho'_s(K\cap (-r,r))=0$ for some $r>0$ (so the point
mass at $x_1=0$ has been removed). We will
handle the weak-$*$ topology in the same way as in the proof of Theorem \ref{T4.6}:
We assume that we are given $\epsilon>0$ and disjoint open
intervals $I_1,\ldots, I_N$ whose endpoints
are not point masses of $\rho$, and our task is to achieve that
\begin{equation}
\label{estnuplus}
\left| \nu'_+(I) - \nu_+(I) \right| < \epsilon ,
\end{equation}
for these intervals
and, of course, $\rho'_s(K\cap (-r,r))=0$ for some $r>0$. As will become clear later on,
we will obtain \eqref{estnuplus} quite easily for those intervals that are at some distance
from $x=0$. So we'll focus on the interval that contains $0$; call this interval $I$.

With these preliminaries out of the way, we are now ready for the main part
of the proof. We will first focus on the situation where gaps accumulate at $0$ from both sides;
equivalently, $0\in K$ is not an endpoint of a gap. The easier alternative cases will be discussed later.

We choose $A,B>0$ so small that $(-A,B)\subset I$,
\begin{equation}
\label{4.7}
\rho((-A,B)\setminus\{ 0\})<\epsilon ,
\end{equation}
and
\begin{equation}
\label{4.6}
\int_{(-A,B)} \frac{|\xi(t)-\chi_{(0,\infty)}(t)|}{|t|}\, dt < \epsilon .
\end{equation}
This can be done because we know that the integral from \eqref{4.6},
extended over $(-1,1)$, say, is finite.
See again \cite[pg.\ 201]{MP} or \cite[Lemma 2.4]{PolRem}.
It will also be convenient to choose $-A$ as the left endpoint of a gap and, similarly,
$B$ as the right endpoint of a (different) gap.

Notice that since $\xi=1/2$ on $K$,
\eqref{4.6} in particular shows that
\begin{equation}
\label{4.1}
\int_{(-A,B)\cap K} \frac{dt}{|t|} < 2\epsilon .
\end{equation}
Now for $0<d<A,B$ (typically $d\ll A,B$), we define
a new $\xi$ function $\xi_d\in X(K)$ as follows:
Let $\xi_d=\xi$ on $(-A,B)^c$.
If $(a,b)$ is a gap of $K$ that is contained in $(-d,d)$, jump from $0$ to $1$ at the center of this
gap; in other words, $\mu=(a+b)/2$. If $d\in(a,b)$ for some gap $(a,b)\subset K^c$,
proceed similarly, but now only the part of $(a,b)$ inside $(-d,d)$ counts: put $\mu=(a+d)/2$.
Use the same procedure at $-d$. Finally, put $\xi_d=1$ on
$(d,B)\setminus K$ and $\xi_d=0$ on $(-A,-d)\setminus K$ and $\xi_d=1/2$ on $K$.

Then $\xi_d\in X(K)$, and, as usual, $\xi_d$
defines a measure $\rho_d$ via \eqref{2.5}, \eqref{2.6}. We claim that
\begin{equation}
\label{estrhoac}
\rho_{d,ac}((-A,B)) \le C\epsilon^{1/2} ,
\end{equation}
for some constant $C>0$. Here and in the remainder of this proof, by a \textit{constant }we mean
a number that is independent of $\epsilon,d,A,B$. It may depend on the set $K$ and the measure $\rho$.
We will also apply the usual convention that the value of a constant may change from one expression
to the next, even though we use the same symbol $C$ for these different constants.

The absolutely continuous part of $\rho_d$ is supported by $K$, and since $\xi_d=1/2$
on $K$, its density is given by $(1/\pi)\textrm{Im }H_d(x) = (1/\pi)|H_d(x)|$.
By \eqref{2.5}, this is equal to
\[
\frac{x+R}{\pi}\exp \left( \lim_{y\to 0+}
\int_{-R}^R\frac{t-x}{(t-x)^2+y^2}\, \xi_d(t)\, dt \right) .
\]
For almost every $x\in\R$, the limit in the exponent (exists and) is equal to the Hilbert transform
\[
(T\xi_d)(x) = \lim_{y\to 0+} \int_{|t-x|>y} \frac{\xi_d(t)\, dt}{t-x} .
\]
Here, the integral is only extended over $t\in (-R,R)$; we will use similar conventions
throughout this paper whenever integrals of $\xi$ functions are involved.

Observe that for almost every $x\in K$, $(T\xi_d)(x)$ goes up if we replace $\xi_d$
by the constant function $1/2$ on $(-d,d)$.
To show this, consider the effect of replacing $\xi_d$ by the constant $1/2$ on a single gap.
It is then obvious from the definition of $(T\xi_d)(x)$ that this quantity goes up at all $x$
not from this gap. To handle the limiting process that is involved here when
we replace $\xi_d$ by $1/2$ on infinitely many
gaps, it suffices to recall that $T$ is a continuous (in fact, unitary) map on $L^2(\R)$.

Similarly, for $-d<x<d$, the Hilbert transform will also go up if we replace $\xi_d$ by $1$
on $(d,R)$ and by $0$ on $(-R,-d)$. For this modified function (call it $\zeta$)
and $x\in (-d,d)$, we have that
\[
(T\zeta)(x) =\lim_{y\to 0+} \frac{1}{2} \int_{(-d,d)\cap \{ |t-x|>y\} } \frac{dt}{t-x}
+ \int_{d}^R \frac{dt}{t-x} ,
\]
and from this we obtain that the density of $\rho_{d,ac}$ satisfies
\[
\frac{d\rho_d}{dx} \le \frac{R^2-x^2}{\pi \sqrt{(d-x)(x+d)}}
\]
on $x\in (-d,d)$. Hence
\begin{equation}
\label{4.2}
\rho_{d,ac}((-d,d)) \le C \int_{K\cap (-d,d)} \frac{dx}{\sqrt{d^2-x^2}} .
\end{equation}
Now \eqref{4.1} clearly implies that $|K\cap (-d,d)| < 2d\epsilon$, and
the integrand from \eqref{4.2} becomes largest close to the endpoints $x=\pm d$,
so we can further estimate \eqref{4.2} as follows:
\[
\rho_{d,ac}((-d,d)) \le C \int_{(1-\epsilon)d}^d \frac{dx}{\sqrt{d^2-x^2}} = C
\int_{1-\epsilon}^1 \frac{ds}{\sqrt{1-s^2}} \le C\epsilon^{1/2} .
\]
This establishes part of \eqref{estrhoac}. Let us now take a look at $\rho_{d,ac}((d,B))$.
As before, we can replace $\xi_d$ by $1/2$ on $(-d,d)$, set it equal to $0$ to the left of
$-d$ and equal to $1$ to the right of $B$, and the density of $\rho_{d,ac}$ will only
go up on $(d,B)$ because the Hilbert transform has this monotonicity property.
Call this modified Krein function $\zeta$, and let $\mu$ be the associated
measure. We will now estimate $\mu_{ac}((d,B))$ by introducing a final modification. We
subdivide $I=(d,B)$ into intervals $I_n=(q^nd,q^{n+1} d)$, where $n=0,1,\ldots , N$ and
$2\le q\le 4$. Let
$s_n=|K\cap I_n|$ and
\[
\zeta_1(x) = \begin{cases} 1/2 & q^n d < x <q^n d+s_n \\
1 & q^n d + s_n < x < q^{n+1}d \\
\zeta(x) & x\notin (d,B)
\end{cases} .
\]
Recall that $\zeta=1$ on $(d,B)\setminus K$ and $\zeta=1/2$ on $K$. This new function
$\zeta_1$ is a function of the same type, but we shifted the part of $K$
inside $I_n$ to the very left of this interval.
We first claim that $\mu(I^c)$ can only decrease under
this change. Again, this follows by comparing Hilbert transforms: it is clear
that $(T\zeta_1)(x)\le (T\zeta)(x)$ if $x\notin I$ (and both limits exist).
This implies that $\mu_{1,ac}(I^c)\le \mu_{ac}(I^c)$, by comparing densities, as
explained above. To obtain the same conclusion for the singular parts, we make
use of the formula
\[
d\rho_s = \lim_{\lambda\to\infty} \frac{\pi}{2}\lambda\chi_{\{ |H|>\lambda\} }(x)\, dx .
\]
See \cite[Theorem 1]{Pol2} and also \cite[Section 9.7]{CMT}
(for the disk version of this statement).

On the other hand, $\|\zeta-\zeta_1\|_1 < B$, so the following Lemma
will make sure that $|\mu(\R)-\mu_1(\R)|<\epsilon$, if $B>0$ was chosen small enough initially.
\begin{Lemma}
\label{L4.1}
Fix $R>0$. Then,
for any $\epsilon>0$, there exists $\delta>0$ so that
\[
|\rho_1(\R)-\rho_2(\R)|<\epsilon
\]
if $\rho_1,\rho_2$ are the measures of two $H$ functions, as in \eqref{2.5}, \eqref{2.6},
whose $\xi$ functions satisfy $\|\xi_1-\xi_2\|_{L^1(-R,R)} < \delta$.
\end{Lemma}
\begin{proof}[Proof of Lemma \ref{L4.1}]
As usual, the map $\xi\mapsto \rho$ is continuous if we use the weak-$*$
topology for both the measures $\xi\, dx$ and the $\rho$'s. The domain $\{ \xi\, dx\}$
is compact, so the map is in fact uniformly continuous, and clearly the $L^1$ norm controls
a distance that generates the weak-$*$ topology.
\end{proof}
Let us summarize: $\mu_1(I^c)\le \mu(I^c)$ and $\mu(\R)<\mu_1(\R)+\epsilon$, hence
$\mu_{ac}(I)=\mu(I)<\mu_1(I)+\epsilon$, and thus it suffices to estimate $\mu_1(I)=\mu_{1,ac}(I)$.
This is done by a calculation, which is similar to the one we used above.
First of all, notice that
\begin{equation}
\label{4.3}
\sum_{n=0}^N \frac{s_n}{q^{n+1}d} \le \int_{(d,B)\cap K} \frac{dt}{t} < 2\epsilon .
\end{equation}
Let's now look at $\mu_1(I_n)=\mu_1((q^nd, q^nd+s_n))$
for fixed $n\ge 1$. If we had $\zeta_1=1$ on all of
$I\setminus I_n$, that would lead to a density of the form
\[
\frac{R^2-x^2}{\pi \sqrt{x^2-d^2}} \left( \frac{x-q^n d}{q^n d+s_n-x}\right)^{1/2}
\]
on $x\in (q^nd,q^nd+s_n)$. Clearly, this expression can be estimated by
\begin{equation}
\label{4.4}
\frac{C}{q^nd} \left( \frac{x-q^n d}{q^n d+s_n-x}\right)^{1/2} .
\end{equation}
This might get bigger by a factor of
\[
\left( \prod_{j=0}^{n-1} \frac{(q^n-q^j)d}{(q^n-q^j)d-s_j} \right)^{1/2} =
\left( \prod_{j=0}^{n-1} \left( 1 + \frac{s_j}{(q^n-q^j)d-s_j} \right) \right)^{1/2} ,
\]
due to the presence of intervals where $\zeta_1=1/2$ to the left of $I_n$. However,
$q^n-q^j\ge q^j$ for $j<n$, and $s_j<q^jd/2$ if $\epsilon<1/16$, by \eqref{4.3}, so \eqref{4.3}
now shows that there is a uniform bound
on this factor, independent of $n\in\N$, $B\le 1$ and $d\le d_0$.
So we may work with \eqref{4.4} after all.
By integrating, it then follows that
\[
\mu_1(I_n) \le \frac{C}{q^nd} \int_0^{s_n} \left( \frac{t}{s_n-t} \right)^{1/2}\, dt =
C \frac{s_n}{q^nd} \quad\quad (n\ge 1) .
\]
For $n=0$, similar reasoning applies and yields the (worse) bound $\mu_1(I_0)\le C (s_0/d)^{1/2}$.
We now see from \eqref{4.3} that $\mu_1(I)\le C_1\epsilon^{1/2}+C_2\epsilon$, as desired.
Since we can of course apply similar arguments to estimate $\rho_{d,ac}((-A,-d))$, we have
now established \eqref{estrhoac}.

Our next goal is to show the following: If $x\notin [-A,B]$ (and the limits
defining $H(x)$, $H_d(x)$ exist), then
\begin{equation}
\label{4.61}
|H_d(x)|\le (1+C\epsilon) |H(x)| .
\end{equation}
Since $\ln |H_d/H|= T(\xi_d-\xi)$, it suffices to compare the Hilbert transforms.
Suppose that $x>B$, say. When going from $\xi$ to $\xi_d$, we only changed
$\xi$ on $(-A,B)$, and the Hilbert transform at $x>B$ will only get smaller when we increase
$\xi$ on this interval. Only
on $(-A,d)$ could $\xi_d$ perhaps be smaller than $\xi$. To estimate the possible effect
of this change, let $M=\{x\in (-A,d): \xi(x)\not= 0 \}$. Clearly, if $x\notin M$, then
$\xi_d(x)\ge\xi(x)$, so we can focus on $M$.
Recall that $\xi$ only takes the values $0,1/2,1$, so \eqref{4.6} shows that
$|M|<3\epsilon A$, at least if we only consider $d\le\epsilon A$.
On the other hand, if also $d<B/2$, say, then the Hilbert transform
at an $x>B$ will not increase by more than $2|M|/B$ when going from $\xi$ to $\xi_d$.
If now $A\le 10B$, say, then this is bounded by $60\epsilon$, and, as explained above,
\eqref{4.61} follows.

If $A>10B$, we introduce
\begin{align*}
m_0 & = |M\cap (-B,d)| ,\\
m_n & = |M\cap (-q^nB,-q^{n-1}B)| , \quad\quad n=1,\ldots, N ,
\end{align*}
with $2\le q\le 4$, and then proceed similarly. Now \eqref{4.6} shows that
\[
\sum_{n=0}^N \frac{m_n}{q^nB} < C\epsilon ,
\]
at least if we again insist that $d$ is taken sufficiently small, say $d\le\epsilon$.
However, this sum also bounds the possible increase of the Hilbert transform at an $x>B$,
so we obtain \eqref{4.61} in this case also.

We also saw earlier that \eqref{4.61} implies a corresponding (local) bound
on the measures, by comparing separately the absolutely continuous and singular parts
of $\rho$ and $\rho_d$. So we have that
$\rho_d(S)\le (1+C\epsilon)\rho(S)$ for all Borel sets $S\subset (-A,B)^c$.

Let us summarize what we have achieved so far; for convenience,
we adjust the constants here. Given $\epsilon>0$ and
an open interval $I$ with $\rho(\partial I)=0$ and $0\in I$, we have constructed a family
of new Krein functions $\xi_d$ that agree with $\xi$ outside the interval $(-A,B)\subset I$.
Here, $A,B>0$ and $0<d\le d_0$ can be chosen as small as we wish; in fact, it will usually
be necessary to take these quantities small enough for the following statements to be true.
Recall also that we arrive at these small values in a two step procedure: We first take
$A,B>0$ sufficiently small, in response to the value of $\epsilon>0$ that was given to us.
Then, in a second step, we pick an appropriate value of $d_0>0$, which will typically be
much smaller still (at the very least $d_0\lesssim \epsilon \min\{ A,B\}$).
This means that we're not allowed to decrease $A,B$ once a range for $d$ has been specified,
but it is permitted to make $d_0$ smaller, according to our needs, while keeping $A,B$ fixed.
\begin{Lemma}
\label{L4.2}
The associated measures $\rho_d$ have no singular part on $K\cap (-A,B)$, and they
have the following additional properties, for small enough $A,B,d_0>0$ and
for all $0<d\le d_0$:\\
(a) $|\rho_d(I)-\rho(I)| < \epsilon$;\\
(b) $\rho_{d,ac}((-A,B)) < \epsilon$;\\
(c) $\rho_d(S)\le (1+\epsilon)\rho(S)$ if $S\cap (-A,B)=\emptyset$
\end{Lemma}
\begin{proof}
We will again compare $\rho_d$ with the measure that is obtained when we replace
$\xi_d$ by the constant value $1/2$ on $(-d,d)$.
Denote the corresponding $H$ function by $H_1$. We claim that then
\begin{equation}
\label{4.81}
\liminf_{y\to 0+} \frac{|H_1(x+iy)|}{|H_d(x+iy)|} > 0
\end{equation}
for all $x\in K$. This follows because the logarithm of this fraction differs
from the corresponding truncated Hilbert transform $T_y(\xi_1-\xi_d)(x)$ by at most
a fixed constant; here, $T_y$ is defined as
\[
(T_y g)(x) = \int_{|t-x|>y} \frac{g(t)\, dt}{t-x} .
\]
The Hilbert transform
goes up when replace $\xi_d$ by $1/2$ on a gap. We used this important observation earlier,
and the truncation does not seriously affect
the argument; it introduces another constant.

Now we compare \eqref{4.81} with the
result from \cite{Pol1} that if $\rho_1 = f\rho_{d,s} + \sigma$, where $\sigma\perp\rho_{d,s}$, then
\[
\lim_{y\to 0+} \frac{H_1(x+iy)}{H_d(x+iy)} = f(x)
\]
for $\rho_{d,s}$-almost every $x\in\R$. So if we had $\rho_{d,s}((-d,d)\cap K)>0$, then \eqref{4.81}
would force $\rho_1$ to have a singular part on $K$ also, but clearly $\rho_1$ is purely
absolutely continuous on $(-d,d)$. Thus $\rho_{d,s}((-d,d)\cap K)=0$.

The Krein function avoids one of the values $0$, $1$ on $(-A,-d)$ and $(d,B)$,
so $\rho_{d,s}$ gives zero weight to these sets, and clearly $\pm d$ can not be
point masses for the chosen arrangement $\xi_d$. Hence $\rho_{d,s}((-A,B)\cap K)=0$.

Part (a) follows from the fact that $\|\xi_d-\xi\|_1\le A+B$, and this
$L^1$ norm controls the distance of $\rho_d$ and $\rho$ in a metric that generates
the weak-$*$ topology, so we just need to take $A,B>0$ small enough. See also the
proof of Lemma \ref{L4.1} for this argument. Part (b) is \eqref{estrhoac},
and part (c) was discussed in the paragraphs preceding the formulation of the Lemma.
\end{proof}
Now return to \eqref{4.9b} and our goal \eqref{estnuplus}; $\nu'_+$
will be chosen as a measure $\nu_{d,+}$ for small $d>0$ and suitable parameters
$\sigma_j(d)$, $g_j(d)$. We can come close to
$\nu_+(I)$ with finite sums in \eqref{4.9b}. More precisely, we pick $N\in\N$
so that, after relabeling if necessary, the expression
\begin{equation}
\label{4.68}
\frac{1}{2}\rho(I) + \left( g_1-\frac{1}{2} \right) v_1 + \sum_{j=2}^N \left(g_j-\frac{1}{2}
\right) v_j + \sum_{j=1}^N \left( \sigma_j - \frac{1}{2} \right) w_j
\end{equation}
differs from $\nu_+(I)$ by not more than $\epsilon$. Here, we of course assume
that $x_j,\mu_j\in I$ for
$j\le N$, and there may be (infinitely many) other point masses in $I$, but these have
total mass less than $\epsilon$. So now our task is to make sure that $\nu_{d,+}(I)$
is close to \eqref{4.68}.

We first demand that $A,B$ are so small that all these $x_j$ (for $2\le j\le N)$
and $\mu_j$ (for $1\le j\le N$) are well separated from
$(-A,B)$, say
\[
A+B < \frac{1}{D}\, \min\{ |x_j|, |\mu_j| \} ,
\]
for a large constant $D$. Then the argument that was used to establish
Lemma \ref{L4.2}(c) also shows that if we take $D$ large enough here (equivalently:
$A,B$ small enough), then
\[
1- \frac{\epsilon}{N} \le \frac{v_j(d)}{v_j} \le 1+ \frac{\epsilon}{N}
\]
for $j=2,\ldots, N$. Here,
$v_j(d)=\rho_d(\{ x_j\})$. Again, this follows because the Hilbert transform of $\xi$
doesn't change much at $x=x_j$ when we pass to $\xi_d$. A similar two-sided estimate holds
for $w_j(d)/w_j$ for $j=1,\ldots, N$, where $w_j(d)=\rho_d(\{ \mu_j \})$.

If we now take
$g_j(d)=g_j$ and $\sigma_j(d)=\sigma_j$ for $j\le N$, then the corresponding contributions
to $\nu_{d,+}$ will differ from the last two sums from \eqref{4.68} by at most $\epsilon$,
and this will be true uniformly in $d\le d_0$. Moreover,
by Lemma \ref{L4.2}(a), $\rho_d(I)$ will be close to $\rho(I)$ for all small $d$.

The other point masses of $\rho_d$ on $I\setminus (-A,B)$ (if any) can be controlled
with the help of Lemma \ref{L4.2}(c). Indeed, we immediately obtain that
\[
\sum_{j>N} (v_j(d)+w_j(d)) < (1+\epsilon)\epsilon < 2 \epsilon ;
\]
here, the sum is really taken only over those $j$ for which the corresponding
point mass (that is, $x_j$ or $\mu_j$) lies in $I\setminus (-A,B)$.

So, to finish the proof of our claim that
we can get close to \eqref{4.68} with $\nu_{d,+}(I)$, it just remains to show that $d\in (0,d_0]$
and the $\sigma_j(d)$ for $|\mu_j(d)|<d$ can be chosen so that
\begin{equation}
\label{4.82}
\left| \left( g_1 - \frac{1}{2} \right)v_1 - \sum_{|\mu_j(d)|<d} \left( \sigma_j(d)-\frac{1}{2}
\right) w_j(d) \right| < C\epsilon .
\end{equation}
Recall that the point masses from this latter sum make up the complete singular part of
$\rho_d$ on $(-A,B)$. So, by Lemma \ref{L4.2}(b),
\begin{equation}
\label{4.83}
\rho_d((-A,B)) - \epsilon \le \sum_{|\mu_j(d)|<d} w_j(d) \le \rho_d((-A,B)) .
\end{equation}
Moreover, by combining parts (a) and (c) of Lemma \ref{L4.2}, we see that
$\rho_d((-A,B))\ge \rho((-A,B)) - C\epsilon$, and, in particular,
\begin{equation}
\label{4.84}
\rho_d((-A,B))\ge v_1 - C\epsilon .
\end{equation}
We will now describe our choice of the parameters $\sigma_j(d)\in \{ 0,1\}$ from \eqref{4.82}
by specifying the index set $M$ on which $\sigma_j(d)=1$. Then clearly
\[
\sum_{|\mu_j(d)|<d} \left( \sigma_j(d)-\frac{1}{2} \right) w_j(d)
= \sum_{j\in M} w_j(d) - \frac{1}{2} \sum_{|\mu_j(d)|<d} w_j(d) .
\]
So, taking \eqref{4.83}, \eqref{4.84} into account, we now see that it suffices to
show that for arbitrary $0\le\alpha\le 1$, there exists a choice of $d>0$ and $M$ that makes
\[
\left| \sum_{j\in M} w_j(d) - \alpha\rho_d((-A,B)) \right| < 2\epsilon .
\]
Now for fixed $d=d_0$, we can clearly achieve that
\begin{equation}
\label{4.85}
\sum_{j\in M} w_j(d_0) \ge \alpha\rho_{d_0}((-A,B)) - 2\epsilon ,
\end{equation}
by just taking a large enough index set $M$. We can and will still insist that $M$ be finite.
Now decrease $d$. The $w_j(d)$ are continuous functions, and each of the finitely many
point masses $\mu_j(d)$ with $j\in M$ will eventually
merge into an endpoint of its gap. For example, if $\mu_j(d)>0$, then $\mu_j(d)=(a_j+b_j)/2$ is
constant as long as $d\ge b_j$, but then starts decreasing, and
finally $\mu_j(a_j)=a_j$. When this happens, we get $w_j=0$. This follows because
with the chosen configuration for $\xi_d$, these endpoints can not be point masses of $\rho_d$.
So there is a smaller positive value $d_1<d_0$ for which the left-hand side of \eqref{4.85} will
become zero. Since $\rho_d((-A,B))$ also is a continuous function of $d$,
some intermediate $d\in [d_1,d_0]$ will work.
This finishes the main part of the argument.

We now discuss the case where $0$ is an endpoint of a gap. The case where there are two such gaps,
say $(-a,0)$ and $(0,b)$, is elementary. In this case, $\xi=\chi_{(0,b)}$ on $(a,b)$.
We can now use a (simplified) version of the Splitting Lemma \ref{L3.3} (where
the small band in the middle has been shrunk to a point)
to handle this case by splitting the point mass at $x=0$ into two nearby point masses in the two gaps,
with the ratio of their weights at our disposal.

Finally, if $(-a,0)$ is a gap, but no gap has $0$ as its left endpoint, then we can run a one-sided
version of the argument given above. Note that $\xi=0$ on $(-a,0)$; otherwise, $x=0$ couldn't be
a point mass of $\rho$. We now take $\xi_d(x)=\xi(x)$ for $x<0$ and modify $\xi(x)$ only
for $x>0$.

This whole procedure can also be used to remove finitely many mass points $x_1,\ldots, x_N\in K$.
We just pick disjoint small neighborhoods of these points and then modify $\xi$ and choose the
parameters $g_j$, $\sigma_j$ in the same way as above on each of these neighborhoods separately.
Note that what we do on one such set will have a negligible effect on what happens on the other
sets because these are separated and thus the Hilbert transform of $\xi$ on one of the sets
will not change much when we modify $\xi$ on the other sets (here it is important that the
sets are chosen to be small compared to their separations).

We have now shown the following: Given $\epsilon>0$ and $J_0\in\RR_0(K)$ with $\rho_{0,sc}(K)=0$,
there exists $J\in\RR_0(K)$ so that $d(J_0,J)<\epsilon$ and $\rho_s(K)<\epsilon$. (We in fact
also know that $\rho_{sc}=0$, but this will not be used.)

To completely remove the singular part of $\rho$ on $K$ here, we proceed as follows.
We know that, as usual,
\begin{align*}
\rho & = \rho_{ac} + \chi_K\rho_s + \sum w_j \delta_{\mu_j} ,\\
\nu_+ & = \frac{1}{2}\rho + \left( g-\frac{1}{2} \right)\chi_K\rho_s
+ \sum \left( \sigma_j - \frac{1}{2} \right) w_j \delta_{\mu_j} ,
\end{align*}
and we would like to come close to $\nu_+$ with a new measure $\nu'_+$ whose corresponding
$\rho'$ measure satisfies $\rho'_s(K)=0$. The argument is quite similar to
the ones given above.

We would like to define new Krein functions $\xi_n$ by putting $\mu_j^{(n)}=b_j$ for $j>n$ and
$\mu_j^{(n)}=\mu_j$ for $1\le j \le n$. However, this runs into minor technical problems, so
the actual definition will be slightly different. We would like to achieve that $\rho_{n,s}(K)=0$,
and this is almost, but not quite, true with our preliminary definition of $\xi_n$.

Notice that $\rho_{n,s}$ does give zero weight to
every open set on which $\xi_n$ avoids the value $1$. So, since $\xi_n=0$ on all but
finitely many gaps, the singular part of $\rho_n$ on $K$ is supported by the set
$\{ a_{j_1(n)},\ldots , a_{j_N(n)} \}$, where this list contains all the left endpoints
of those gaps on which $\xi_n=1$ (if any). We now exclude the possibility
of point masses of this type by preemptively setting
$\mu=a+\delta_n$, with a very small $\delta_n>0$, if $a=a_{j_k(n)}$ is one of these points.
If the $\delta_n$ are chosen small
enough here, then $\|\xi_n-\xi\|_1\to 0$ and thus also $\rho_n\to\rho$ in the weak-$*$ topology.

Next, fix $N_0\in\N$ so that
\begin{equation}
\label{4.73}
\sum_{j>N_0} w_j < \epsilon .
\end{equation}
By relabeling, if necessary, we may assume that $w_j>0$ for $j=1,\ldots ,N_0$ (equivalently,
the corresponding $\mu_j$'s lie in the interiors of their gaps), and we
then define $\sigma_j^{(n)}=\sigma_j$ for $j=1,\ldots, N_0$. To compare $\nu_+$ and $\nu_{n,+}$
in the weak-$*$ topology, we again compare the weights these measures give to fixed intervals
$I$ with $\rho(\partial I)=0$. If $n\ge N_0$, then
\begin{equation}
\label{4.71}
\nu_{n,+}(I) = \frac{1}{2}\rho_n(I) + \sum_{j\le N_0;\mu_j\in I} \left( \sigma_j-\frac{1}{2}
\right) w_j^{(n)} + \sum_{j>N_0; \mu_j^{(n)}\in I} \left( \sigma_j^{(n)}-\frac{1}{2}
\right) w_j^{(n)} ,
\end{equation}
and here $\rho_n(I)\to\rho(I)$, $w_j^{(n)}\to w_j$ as $n\to\infty$ for $j=1,\ldots, N_0$.
Since also $\rho_s(K)<\epsilon$, this implies that the first two terms from the right-hand side
of \eqref{4.71} will differ from $\nu_+(I)$ by not more than $\epsilon$ for all large $n$.
So we have to show that the last sum from \eqref{4.71} can be made small by choosing
the corresponding $\sigma_j^{(n)}$ appropriately. This follows as in the proof of
Theorem \ref{T4.6} from the fact that
\begin{equation}
\label{4.72}
w_j^{(n)}< \epsilon \quad\quad\quad (j>N_0, n\ge n_0) ;
\end{equation}
here $n_0$ must be taken sufficiently large. Indeed, \eqref{4.72} guarantees
that the last sum from \eqref{4.71} can be changed from positive values
(all $\sigma$'s equal to $1$) to negative values in small steps of size $<\epsilon$,
so there is a configuration for which it will be within distance $\epsilon$ from $0$, as desired.

To prove \eqref{4.72}, assume the contrary. Then we find $w_{j_n}^{(n)}\ge\epsilon$
with $j_n>N_0$ for certain arbitrarily large values of $n$.
We may assume that the corresponding $\mu_{j_n}^{(n)}$
converge, to $x\in\R$, say. Since $\rho_n\to\rho$, this implies
that $\rho(\{ x\})\ge \epsilon$. This is a contradiction because $\rho_s(K)<\epsilon$
and \eqref{4.73} holds, so the only point masses of $\rho$ whose weight could possibly
be $\ge\epsilon$ are $\mu_1,\ldots,\mu_{N_0}$, but these are in the interiors of their
gaps, so $\mu_{j_n}^{(n)}$ with $j_n>N_0$
can certainly not converge to one of these.
\end{proof}
\section{Sets of measure zero}
On a zero measure set, the condition of being reflectionless becomes vacuous, and this
changes the character of our results. Formally, they remain true, though, and we now discuss
them one by one in this new situation. Since this discussion is somewhat removed from
our main topic in this paper, we will be extremely sketchy here.

If $|B|=0$, then Theorem \ref{T1.1} says that any $J\in\J$ can be approximated by periodic
$J_n\in\J_R$ whose spectra satisfy $|\sigma(J_n)|\to 0$. We can do this directly, as follows:
Let $a(j),b(j)$ be the coefficients of $J$, and put $a_n(j)=a(j)$, $b_n(j)=b(j)$ for
$|j|\le n$. This already guarantees that $J_n\to J$, provided the remaining coefficients are
chosen so that there is a uniform bound on their size. We now put $a_n(n+1)=0$, $b_n(n+1)=0$
(the value of $b$ is irrelevant here), and extend periodically. Then $J_n$ is an infinite
sum of copies of a finite-dimensional operator, so $\sigma(J_n)$ is a finite set.
This proves the trivial measure zero version of Theorem \ref{T1.1}. If we want to,
we can in fact avoid a zero coefficient $a_n(j)$ here;
we then just assign a very small positive value to $a_n(n+1)$.

Proposition \ref{P1.1} remains true for $|K|=0$, with the same proof. Theorem \ref{T1.3}
also continues to hold unless $K$ consists of exactly two points, simply because its hypothesis
is never satisfied if $|K|=0$ and $K$ has at least three points.

As for the next set of results, our treatment from above still applies if $|K|=0$.
In fact, as we will see, things become much easier now.
If $|K|=0$ and any $J\in\RR_0(K)$ is given, then we can use the methods from Section 5
to produce, for any $\epsilon>0$, a $J'\in\RR_0(K)$ so that $\rho'_{pp}(K)<\epsilon$ and
$d(J,J')<\epsilon$. Having done that, we can now run the construction from the last part of
Section 5 to find a $J''\in\RR_0(K)$ with $\rho''(K)=0$ and $d(J',J'')<\epsilon$.
The key new feature here is the fact that whether or not $\rho'_{sc}(K)$ was small to start with,
this measure can never be transformed into absolutely continuous measure. The corresponding weight
can only (approximately) go into the point masses in the gaps if $|K|=0$. Much of the analysis
of Section 5 centered around this question, how do we prevent singular measure on $K$ from
becoming absolutely continuous measure on $K$ under a small perturbation, and this problem has
completely disappeared now.

A new technical issue arises here due to the possibility of finitely supported measures
$\rho$, $\nu_+$; in fact, $K$ itself could be a finite set.
Recall from Section 2 that the spectral data $(\xi,\nu_+)$ do \textit{not }uniquely
determine a Jacobi matrix $J$ in this case, so it is not enough to construct such data
that are close to the ones corresponding to the given Jacobi matrix.
However, we can simply observe that a Jacobi matrix with finite spectrum splits into
finite-dimensional blocks, and we can then discuss these blocks separately, using the
techniques outlined above. We leave the matter at that and hope that this very brief
sketch has given an impression of how the following result could be proved.
\begin{Theorem}
\label{T6.1}
Let $K\subset\R$ be a non-empty compact set with $|K|=0$. Recall that
\[
\RR_0(K) = \{ J\in\J : \sigma(J)\subset K \} .
\]
For every $J\in\RR_0(K)$, there are $J_n\in\RR_0(K)$ with $\rho_n(K)=0$, so that
$d(J_n,J)\to 0$.
\end{Theorem}
By following the pattern described at the beginning of Section 5,
this improved version of Theorem \ref{T4.1} in a special case then leads
to improved versions of Theorems \ref{T1.4}, \ref{T1.5}. In fact, the treatment
again becomes much simpler, and one could also make more specific statements.
For example, if a sequence $K_n$ with $\delta(K_n, K)\to 0$ is given and $|K|=0$,
then one could pick finite subsets $F_n\subset K_n$ so that still $\delta(F_n,K)\to 0$.
This means that for any $J$ with $\sigma(J)\subset K$, there exists a sequence
of operators $J_n$ whose spectra are \textit{finite }subsets of the $K_n$, and
$J_n\to J$.

Finally, notice that if $|K|=0$
and $h(K_n,K)\to 0$, then $|K_n|\to 0$, so, as noted in the introduction, $h(K_n,K)\to 0$
implies that $\delta (K_n,K)\to 0$ in this case.

\end{document}